\documentclass[12pt,authoryear]{elsarticle}
\journal{}

\makeatletter
\def\ps@pprintTitle{%
 \let\@oddhead\@empty
 \let\@evenhead\@empty
 \def\@oddfoot{\hfill\thepage}%
 \let\@evenfoot\@oddfoot}
\makeatother

\usepackage{amsmath,amssymb,amsfonts,amsthm,graphicx,appendix}
\usepackage[bookmarksnumbered,colorlinks=true]{hyperref}
\usepackage[labelfont=bf]{caption}

\providecommand{\doi}[1]{\href{https://doi.org/#1}{DOI:#1}}
\usepackage{xurl} 
\renewcommand{\doi}[1]{%
 \href{https://doi.org/#1}{\nolinkurl{DOI:#1}}%
}

\usepackage[OT1]{fontenc}
\usepackage{natbib,enumitem,mathtools,mathrsfs,dsfont}
\usepackage{tikz,float}
\usetikzlibrary{arrows.meta}
\usepackage{geometry}
\newtheorem{theorem}{Theorem}[section]
\newtheorem{lemma}[theorem]{Lemma}

\newtheorem{proposition}[theorem]{Proposition}

\theoremstyle{definition}

\newtheorem{remark}[theorem]{Remark}

\numberwithin{equation}{section}

\newcommand{\N}{\mathbb{N}}

\newcommand{\R}{\mathbb{R}}

\newcommand{\PP}{\mathsf{P}}
\newcommand{\EE}{\mathsf{E}}
\newcommand{\Var}{\mathsf{Var}}
\newcommand{\Cov}{\mathsf{Cov}}
\newcommand{\bb}[1]{\boldsymbol{#1}}
\newcommand{\ind}{\mathds{1}}
\newcommand{\rd}{\mathrm{d}}
\newcommand{\e}{\varepsilon}
\newcommand{\ii}{\mathrm{i}}
\renewcommand{\Re}{\mathrm{Re}\, }
\renewcommand{\Im}{\mathrm{Im}\, }
\newcommand{\OO}{\mathcal{O}}
\newcommand{\oo}{\mathrm{o}}
\newcommand{\leqdef}{\vcentcolon=}

\newcommand{\sig}{\sigma}
\newcommand{\logp}{\log_{+}}

\begin{document}

\begin{frontmatter}

\title{Overlap distributions under the Gibbs measure of a Haar model \\of $\log|\zeta|$ on and near the critical line}

\author[a1]{Fr\'{e}d\'{e}ric Ouimet}

\address[a1]{Universit\'{e} du Qu\'{e}bec \`{a} Trois-Rivi\`{e}res, Canada. \ead{frederic.ouimet2@uqtr.ca}\vspace{-7mm}}

\begin{abstract}
We study the Gibbs measure of $H(h) = \Re\sum_{p^{j} \leq X}U_{p}^{j}/(j p^{j(\sig + \ii h)})$, a random field modeling large values of the log-modulus of the Riemann zeta function on and near the critical line, where the $U_{p}$ are independent and uniformly distributed on the complex unit circle. Writing $t = \log \log T$, the cutoff $X$ and real part $\sig$ may vary with $T$ subject to $T^{1/t^{\Phi}} \leq X \leq T^{1/t^{\phi}}$ and $0 \leq \sig - 1/2 \leq C_{\sig}t^{\Phi}/\log T$, for fixed $\Phi \geq \phi > 1/2$ and $C_{\sig} > 0$. On $I_{\theta} = [ - (\log T)^{\theta},(\log T)^{\theta}]$, define the Gibbs measure $\smash{G_{\beta,\sig,T}^{\mathrm{H}}(\rd{h}) \leqdef e^{\beta H(h)}\rd{h}/\int_{I_{\theta}}e^{\beta H(h')}\rd{h'}}$. For every fixed $\theta > - 1$ and inverse temperature $\beta > 0$, the limiting joint distributions of correlations (overlaps) among points sampled from this measure are determined after averaging over the Haar phases. At and below $\beta_{c}(\theta) = 2\sqrt{1 + (\theta \vee 0)}$, the limiting two-overlap distribution is concentrated at $q_{0} = |\theta \wedge 0|$. Above this threshold, it assigns masses $\beta_{c}(\theta)/\beta$ and $1 - \beta_{c}(\theta)/\beta$ to $q_{0}$~and~$1$, respectively, and the limiting joint overlap distribution is described by a Poisson--Dirichlet process of parameter $\beta_{c}(\theta)/\beta$. Approximate Ghirlanda--Guerra identities are shown to hold uniformly for bounded deterministic functions of the sampled positions.
This paper is the first in a series devoted to extending these results to $\log|\zeta|$ near the critical line.
\end{abstract}


\end{frontmatter}

\section{Introduction}\label{sec:intro}

The Haar field studied here is a probabilistic model for large values of the log-modulus of the Riemann zeta function, $\log|\zeta|$; see \citet{arXiv:1304.0677} for motivation and the pedagogical survey of \citet{arXiv:1601.00582} for background on log-correlated fields and their connection to $\log|\zeta|$. The construction starts from the logarithmic Euler expansion
\[
\log|\zeta(s)| = \Re\sum_{p}\sum_{j \geq 1}\frac{1}{j p^{js}}, \qquad \Re(s) > 1.
\]
If $\tau$ is uniformly distributed on $[T,2T]$, the phases $p^{-\ii\tau}$ attached to any fixed finite collection of distinct primes converge jointly in distribution to independent uniform variables on the complex unit circle as $T \to \infty$; see \citet[Lemma~1.2]{arXiv:1904.08204}. This motivates truncating the expansion at $p^{j} \leq X$ and replacing the prime phases by independent variables $U_{p}$ with the (normalized) Haar measure on the unit circle. For $s = \sig + \ii\tau + \ii h$, this yields the full truncated Haar field $H$, defined by
\[
H(h) = \Re\sum_{p^{j} \leq X}\frac{U_{p}^{j}}{j p^{j(\sig + \ii h)}}, \qquad h\in I_{\theta} \leqdef [ - (\log T)^{\theta},(\log T)^{\theta}].
\]

For the related prime-only model on bounded intervals of the critical line ($\theta = 0$),
\[
X_{h}^{\mathrm{rand}} \leqdef \Re\sum_{p \leq T}\frac{U_{p}}{p^{1/2 + \ii h}}, \qquad h \in [0,1],
\]
\citet{MR3619786} determined the first two orders of the maximum. The Gaussian coupling of \citet{MR4164452}, combined with the tightness theorem of \citet{MR3272323}, yields tightness of $\max_{h \in [0,1]}X_{h}^{\mathrm{rand}} - \log \log T + (3/4)\log \log \log T$. The same coupling gives non-Gaussian multiplicative chaos limits for suitably normalized exponential measures in the subcritical and critical regimes. For $0 < \beta < 2$, \citet{arXiv:1906.08573} showed that the Lebesgue measure of $\{h \in [0,1] : X_{h}^{\mathrm{rand}} > (\beta/2)\log \log T\}$, divided by its expectation, converges in probability to the total mass of the corresponding subcritical chaos measure. For the Gaussian-prime counterpart of $X^{\mathrm{rand}}$, \citet{ADH_2021} and \citet{Chang2024} obtained second-order asymptotics for maxima, matching right-tail bounds, and critical-moment estimates over intervals of varying length, including $\theta = (\log{\log T})^{-\alpha}$ with $\alpha \in (0,1)$.

At low temperature ($\beta > 2$), \citet{arXiv:1706.08462} identified the limiting two-overlap distribution of $X^{\mathrm{rand}}$ as the two-point mixture $(2/\beta)\delta_{0} + (1 - 2/\beta)\delta_{1}$, and \citet{MR3841407} established the approximate Ghirlanda--Guerra (GG) identities and the Poisson--Dirichlet description of the joint overlap distributions in the limit $T \to \infty$. Related universal Poisson--Dirichlet limits for log-correlated Gaussian fields were established by \citet{GangulyNam2025}. However, $X^{\mathrm{rand}}$ and the Haar field $H$ are non-Gaussian, so their results do not apply. The independent prime contributions nevertheless allow a derivation of the GG identities through approximate integration by parts.

Let $T > 0$ tend to infinity, write $t \leqdef \log \log T$, and fix the interval-length exponent $\theta > - 1$ and the inverse temperature $\beta > 0$. For fixed constants $\Phi \geq \phi > 1/2$ and $C_{\sig} > 0$, let the real part $\sig = \sig(T)$ and the prime-power cutoff $X = X(T)$ be deterministic functions satisfying, for all sufficiently large $T$,
\begin{equation}\label{eq:fixed.parameters}
T^{1/t^{\Phi}} \leq X \leq T^{1/t^{\phi}}, \qquad 0 \leq \sig - \frac{1}{2} \leq \frac{C_{\sig}t^{\Phi}}{\log T}.
\end{equation}
This range includes $\sig = 1/2$ and, for example, the choices $X = T^{1/t^{4}}$ and $\sig = 1/2 + \delta t^{6}/\log T$ for any fixed $\delta \geq 0$, by taking $\Phi = 6$, $\phi = 4$, and $C_{\sig} \geq \delta$. The dependence on $X$ is suppressed in the notation for the Haar field, its Gibbs measure, and its overlaps. The Haar field is sampled on the interval $I_{\theta}$ according to its Gibbs measure
\begin{equation}\label{eq:haar.field}
G_{\beta,\sig,T}^{\mathrm{H}}(\rd{h}) = \frac{\exp(\beta H(h))\rd{h}}{\int_{I_{\theta}}\exp(\beta H(h'))\rd{h'}},
\end{equation}
where $\rd{h}$ denotes integration with respect to Lebesgue measure. Probability and expectation over the Haar phases are denoted by $\PP_{\mathrm{H}}$ and $\EE_{\mathrm{H}}$. The pointwise variance and the correlation, or overlap, are
\[
V_{T} \leqdef \EE_{\mathrm{H}}[H(h)^{2}], \qquad R_{\sig,T}^{\mathrm{H}}(h,h') \leqdef \frac{\EE_{\mathrm{H}}[H(h)H(h')]}{V_{T}}.
\]
The Haar field is stationary by rotational invariance of the independent phases $U_{p}$, so $V_{T}$ is independent of $h$. It is also logarithmically correlated: uniformly in distinct $h,h' \in I_{\theta}$,
\[
\EE_{\mathrm{H}}[H(h)H(h')]
= \frac{1}{2}\sum_{p^{j} \leq X}\frac{\cos(j(h - h')\log{p})}{j^{2}p^{2j\sig}}
= \frac{1}{2}\min\left\{t,\logp\left(\frac{1}{|h - h'|}\right)\right\} + \OO(\log{t}),
\]
as $T \to \infty$, where $\logp{y} = \max\{\log{y},0\}$. Moreover, $V_{T} = t / 2 + \OO(\log{t})$; see Lemma~\ref{lem:covariance.input}.

Sampling $h$'s from the Gibbs measure favors large values of the Haar field, with larger $\beta$ placing greater relative weight on higher values. The joint overlaps describe the correlation structure at the sampled positions and thus provide a way to study the geometry of large values of the Haar field, including the hierarchical clustering associated with ultrametricity; see \citet{MR3628881} for a general theory of approximate ultrametricity for random measures.

The remainder of the paper is organized as follows. Section~\ref{sec:definitions} introduces the additional definitions and notation needed to state the results. Section~\ref{sec:haar} presents the main result and discusses its conclusions. Section~\ref{sec:proofs} contains the supporting lemmas and all proofs.

\section{Definitions and notation}\label{sec:definitions}

\subsection{Parameters and general conventions}\label{subsec:parameters}

Throughout, the deterministic parameters $(\sig,X)$ satisfy \eqref{eq:fixed.parameters}. Define the distance from the critical line, the half-length of $I_{\theta}$, and the effective logarithmic cutoff by
\[
\eta_{\sig} \leqdef \sig - 1/2, \qquad L_{T} \leqdef (\log T)^{\theta}, \qquad \Lambda_{T} \leqdef \min\{\log{X},\eta_{\sig}^{-1}\},
\]
with the convention $\eta_{\sig}^{-1} = \infty$ when $\eta_{\sig} = 0$. In particular, $\Lambda_{T} = \log{X}$ on the critical line. The parameter range \eqref{eq:fixed.parameters} gives
\begin{equation}\label{eq:effective.cutoff.bounds}
\frac{e^{t}}{\max\{1,C_{\sig}\}t^{\Phi}} \leq \Lambda_{T} \leq \frac{e^{t}}{t^{\phi}}, \qquad \log{\log{X}} = t + \OO(\log{t}), \qquad \log{\Lambda_{T}} = t + \OO(\log{t}).
\end{equation}
Also, $I_{\theta} = [ - L_{T},L_{T}]$. The smallest limiting overlap and the critical inverse temperature are
\[
q_{0} \leqdef |\theta \wedge 0|, \qquad
\beta_{c}(\theta) \leqdef
\begin{cases}
2, & \theta \leq 0, \\
2\sqrt{1 + \theta}, & \theta > 0.
\end{cases}
\]
The regimes $\beta < \beta_{c}(\theta)$, $\beta = \beta_{c}(\theta)$, and $\beta > \beta_{c}(\theta)$ are called subcritical (high temperature), critical, and supercritical (low temperature), respectively.

For $n \in \N \equiv \{1,2,\ldots\}$, $\mathrm{Id}_{n}$ is the $n \times n$ identity matrix, and $\bb{1}$ is the column vector of $n$ ones. For any set $A$, $\ind_{A}$ denotes its indicator. For real $a$ and $b$, write $a \wedge b = \min\{a,b\}$, $a \vee b = \max\{a,b\}$, and $a_{+} = \max\{a,0\}$. For $y > 0$, write $\logp{y} = \max\{\log{y},0\}$. In prime sums, $p$ ranges over primes, and $j$ ranges over positive integers when indexing prime powers.

For a positive comparison function $g$, the notation $f = \OO(g)$ or $f \ll g$ means that $|f|$ is bounded by a constant multiple of $g$, while $f = \oo(g)$ means that $f/g \to 0$. For positive $f$ and $g$, the notation $f \asymp g$ means that their ratio is bounded above and below by positive constants. Unless another limit is specified, all limits are taken as $T \to \infty$. Positive finite constants, usually denoted by $c$ or $C$, may change from one occurrence to the next. They may depend on fixed parameters, including $\phi$, $\Phi$, and $C_{\sig}$, and explicitly fixed cutoffs, but not on $T$ or the choice of $(\sig,X)$ within \eqref{eq:fixed.parameters}; subscripts record any additional indicated dependence. Probability and expectation for auxiliary random objects, including the Gaussian replacements, Gaussian forests, and limiting overlap arrays introduced below, are denoted by $\PP$ and $\EE$ when their law is clear from the context.

\subsection{The Haar field, Gibbs sampling, and overlaps}\label{subsec:haar.definitions}

For $0 < \alpha \leq 1$, the Haar field truncated at prime scale $\alpha$ is
\[
H_{\alpha}(h) = \Re \sum_{p^{j} \leq \exp((\log{X})^{\alpha})} \frac{U_{p}^{j}}{j p^{j(\sig + \ii h)}}, \qquad h\in I_{\theta},
\]
so that $H_{1} = H$. Its cross-covariance kernel with the full Haar field is
\begin{equation}\label{eq:haar.covariance.definitions}
K_{\alpha}(h,h') \leqdef \EE_{\mathrm{H}}[H_{\alpha}(h)H(h')].
\end{equation}
In particular, $V_{T} = K_{1}(h,h)$ and $R_{\sig,T}^{\mathrm{H}}(h,h') = K_{1}(h,h')/V_{T}$.

To clarify the role of $\theta$ in the log-correlation structure of $H$, consider only the terms with $j = 1$: higher prime powers contribute $\OO(1)$ to covariances by \eqref{eq:pnt.higher.powers}. When $ - 1 < \theta < 0$, for every fixed $0 < \e < q_{0}$,
\[
\sup_{h \in I_{\theta},\,\log{p} \leq (\log T)^{q_{0} - \e}} |h\log{p}| \leq (\log T)^{-\e} \to 0, \qquad T \to \infty.
\]
Thus $p^{-\ii h} = 1 + \oo(1)$ uniformly over these primes and over $h \in I_{\theta}$, so their spatial phases are approximately constant throughout the interval. The normalized prime-scale coordinate $\alpha$ corresponds to the cutoff $\log{p} \leq (\log{X})^{\alpha}$. Since $\log{\log{X}}/t \to 1$ by \eqref{eq:fixed.parameters}, the prime contributions are approximately common at every fixed scale $\alpha < q_{0}$, whereas scales $\alpha > q_{0}$ produce nontrivial variation in $h$. When $\theta > 0$, the interval contains about $(\log T)^{\theta}$ subintervals of bounded length, and values of the Haar field at positions separated by a distance bounded below by a positive constant have correlation tending to zero, by Lemma~\ref{lem:covariance.input}. The trees in Figure~\ref{fig:prime.scale.regimes} illustrate the leading covariance structure in these two regimes.

\begin{figure}[H]
\centering
\includegraphics[width=\textwidth]{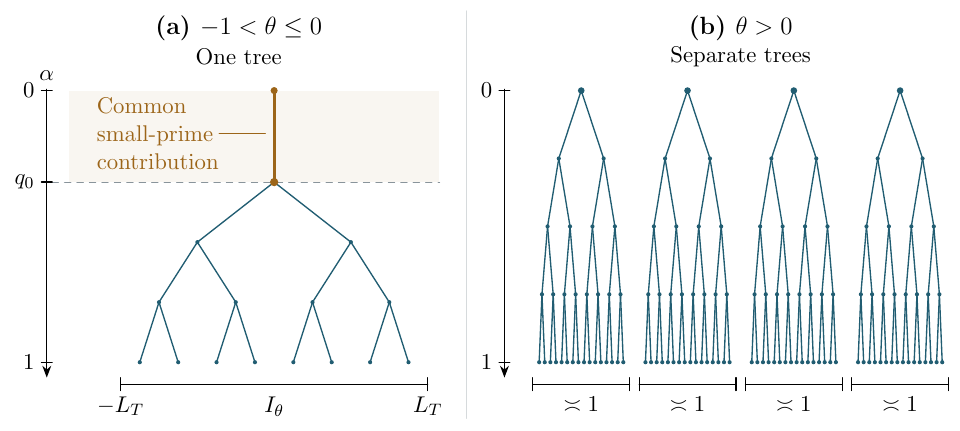}
\caption{Schematic prime-scale description on $I_{\theta}$. The vertical coordinate is $\alpha$, with prime cutoff $\log{p} \leq (\log{X})^{\alpha}$; leaves represent spatial points separated on the scale $(\log T)^{-1}$. For $ - 1 < \theta \leq 0$, the common trunk represents the approximately common prime contributions at normalized scales strictly below $q_{0}$, followed by branching over the remaining depth $1 + \theta$. The common trunk disappears when $\theta = 0$. For $\theta > 0$, about $(\log T)^{\theta}$ bounded subintervals are represented by separate trees of depth $1$, each with about $\log T$ leaves; this depicts the vanishing correlation between values in well-separated subintervals.}\label{fig:prime.scale.regimes}
\end{figure}

The partition function and free energy of the Haar field are
\[
Z_{\beta,\sig,T}^{\mathrm{H}} \leqdef \int_{I_{\theta}}e^{\beta H(h)}\rd{h}, \qquad \frac{1}{t}\log Z_{\beta,\sig,T}^{\mathrm{H}}.
\]
The unnormalized quantity $\log Z_{\beta,\sig,T}^{\mathrm{H}}$ is called the log-partition function. The Gibbs measure in \eqref{eq:haar.field} is a random probability measure on $I_{\theta}$ equipped with its Borel $\sigma$-algebra $\mathscr{B}(I_{\theta})$.

For a probability measure $\mu$ and a bounded measurable function $\Psi$ of $n$ positions, $\mu^{\times n}[\Psi(h_{1},\ldots,h_{n})]$ denotes integration against its $n$-fold product. Write $\bb{h} = (h_{1},\ldots,h_{n})$. Conditional on the Haar phases, the sampled positions $h_{1},\ldots,h_{n}$ are independent with common law $G_{\beta,\sig,T}^{\mathrm{H}}$. The annealed expectation $\EE_{\mathrm{H}}(G_{\beta,\sig,T}^{\mathrm{H}})^{\times n}[\Psi(\bb{h})]$ also averages over the phases. The two-overlap distribution is the law of $R_{\sig,T}^{\mathrm{H}}(h_{1},h_{2})$ under this averaged sampling procedure, and the finite overlap array is $(R_{\sig,T}^{\mathrm{H}}(h_{i},h_{j}))_{1 \leq i,j \leq n}$. Since $R_{\sig,T}^{\mathrm{H}}$ is a correlation kernel, its values belong to $[ - 1,1]$ and its diagonal entries equal $1$.

The deterministic logarithmic distance kernel is
\[
r_{T}(h,h') \leqdef \max\{0,\min\{1,t^{-1}\log(1 / |h - h'|)\}\}, \qquad h \neq h',
\]
with $r_{T}(h,h) = 1$. Lemma~\ref{lem:covariance.input} shows that this kernel approximates the overlap $R_{\sig,T}^{\mathrm{H}}(h,h')$ uniformly in $h,h' \in I_{\theta}$ with error $\OO((\log{t})/t)$.

\subsection{Poisson--Dirichlet weights and the perturbed Haar field}\label{subsec:pd.perturbation}

A Poisson--Dirichlet process $\bb{\xi} = (\xi_{j})_{j \in \N}$ with parameter $\lambda \in (0,1)$ is the decreasing rearrangement of normalized atoms of a Poisson random measure on $(0,\infty)$ with intensity $\lambda v^{-1-\lambda}\rd{v}$. More precisely, if $(\eta_{j})_{j \in \N}$ are these atoms, then $\bb{\xi}$ is the decreasing rearrangement of $(\eta_{j}/\sum_{k}\eta_{k})_{j \in \N}$. Counts in disjoint measurable sets are independent, and the count in a set $A$ is Poisson with mean $\int_{A}\lambda v^{-1-\lambda}\rd{v}$ whenever this integral is finite.

Fix $q_{0} < \alpha < 1$, and let $H^{u} = H + uH_{\alpha}$ be the perturbed Haar field, with multiplier $1 + u$ on the scales below $\alpha$. The perturbation parameter is restricted to $u \in [ - 1/2,1/2]$, but any compact subinterval of $( - 1,\infty)$ containing $0$ in its interior would do. Define the free energy of $H^{u}$ by
\begin{equation}\label{eq:perturbed.free.energy.definition}
f_{T}(u) = \frac{1}{t}\log \int_{I_{\theta}} \exp(\beta(H(h) + uH_{\alpha}(h))) \rd{h}.
\end{equation}

\section{Main result}\label{sec:haar}

Our main result, Proposition~\ref{prop:haar.laws}, establishes three conclusions for the full Haar field $H$ throughout the parameter range \eqref{eq:fixed.parameters}, for every $\theta > - 1$ and $\beta > 0$: the limiting two-overlap distribution, approximate GG identities for bounded measurable functions of the sampled positions, and the limiting joint distribution of the overlaps of every fixed number of sampled positions. These conclusions are closely related: the two-overlap distribution determines the possible limiting overlap values and their probabilities, while the GG identities give consistency relations as further positions are sampled. Together, they lead to the Poisson--Dirichlet description of the limiting joint overlap distribution at low temperature.

The GG identities were introduced by \citet{MR1662161} in the study of the Sherrington--Kirkpatrick model, whose random Hamiltonian is a Gaussian field indexed by configurations in $\{ - 1,1\}^{N}$. For the generalized random energy model (GREM), a Gaussian field with hierarchical covariance, \citet{MR2070334,MR2070335} used these identities and the two-overlap distribution to characterize the limiting Gibbs structure through hierarchical Poisson cascades (called Ruelle's probability cascades). For mixed $p$-spin models, \citet{MR2600075} established GG identities under suitable assumptions on the interactions, and \citet{MR3901654} proved universality results for these identities and spin distributions with respect to the disorder. General background on Gibbs measures, overlaps, and their role in spin glass theory can be found in \citet{MR3052333} and \citet{MR2731561,MR3024566}.

For log-correlated Gaussian fields, two-point limiting overlap distributions and Poisson--Dirichlet descriptions of the limiting joint overlap distributions at low temperature were obtained by \citet{MR3211001} for a nonhierarchical model related to the logarithmic random energy model, by \citet{MR3354619} for the two-dimensional discrete Gaussian free field, and by \citet{MR3539644} for the branching random walk. These descriptions combine the two-overlap distribution with GG identities. For the scale-inhomogeneous discrete Gaussian free field, \citet{MR3731796} established GG identities and identified the limiting overlap structure with that of the GREM. In the non-Gaussian setting of the prime-only field $X_{h}^{\mathrm{rand}}$ mentioned in Section~\ref{sec:intro}, \citet{MR3841407} likewise obtained a Poisson--Dirichlet joint overlap limit by combining GG identities with the two-overlap distribution previously identified by \citet{arXiv:1706.08462}. The following proposition extends these results to $H$.

\begin{proposition}[Overlap distributions and approximate GG identities for the Haar field]\label{prop:haar.laws}
For every fixed $\theta > - 1$ and $\beta > 0$ and every deterministic choice of $(\sig,X)$ satisfying \eqref{eq:fixed.parameters},
\begin{equation}\label{eq:haar.two.overlap}
\EE_{\mathrm{H}} (G_{\beta,\sig,T}^{\mathrm{H}})^{\times 2}[f(R_{\sig,T}^{\mathrm{H}}(h_{1},h_{2}))]
\to \min\left\{1,\frac{\beta_{c}(\theta)}{\beta}\right\}f(q_{0}) + \left(1 - \min\left\{1,\frac{\beta_{c}(\theta)}{\beta}\right\}\right)f(1)
\end{equation}
as $T \to \infty$, for every continuous $f$ on $[ - 1,1]$. For every fixed $n \in \N$ and $1 \leq k \leq n$, the approximate GG identity
\begin{equation}\label{eq:haar.gg}
\begin{aligned}
\EE_{\mathrm{H}} (G_{\beta,\sig,T}^{\mathrm{H}})^{\times(n+1)}[R_{\sig,T}^{\mathrm{H}}(h_{k},h_{n+1})\Psi_{T}(\bb{h})]
&= \frac{1}{n} \EE_{\mathrm{H}} (G_{\beta,\sig,T}^{\mathrm{H}})^{\times 2}[R_{\sig,T}^{\mathrm{H}}(h_{1},h_{2})] \EE_{\mathrm{H}} (G_{\beta,\sig,T}^{\mathrm{H}})^{\times n}[\Psi_{T}(\bb{h})] \\
&\quad + \frac{1}{n}\sum_{\substack{1 \leq j \leq n \\j \neq k}} \EE_{\mathrm{H}} (G_{\beta,\sig,T}^{\mathrm{H}})^{\times n} [R_{\sig,T}^{\mathrm{H}}(h_{k},h_{j})\Psi_{T}(\bb{h})] + \oo(1)
\end{aligned}
\end{equation}
holds uniformly over measurable functions $\Psi_{T}:I_{\theta}^{n} \to [ - 1,1]$. If $\beta \leq \beta_{c}(\theta)$, define the limiting overlap matrix $Q^{\mathrm{high}} \leqdef (1 - q_{0})\mathrm{Id}_{n} + q_{0}\bb{1}\bb{1}^{\mathsf{T}}$. If $\beta > \beta_{c}(\theta)$, let $\bb{\xi}$ have Poisson--Dirichlet parameter $\beta_{c}(\theta)/\beta$, sample conditionally independent group labels $K_{1},\ldots,K_{n}$ with probabilities $\bb{\xi}$, and define the random limiting overlap matrix $Q^{\mathrm{low}} \leqdef (q_{0} + (1 - q_{0})\ind_{\{K_{\ell} = K_{\ell'}\}})_{1 \leq \ell,\ell' \leq n}$. For every continuous function $\psi:[ - 1,1]^{n \times n} \to \R$,
\begin{equation}\label{eq:haar.joint.array}
\lim_{T \to \infty}\EE_{\mathrm{H}}(G_{\beta,\sig,T}^{\mathrm{H}})^{\times n}\left[\psi\left((R_{\sig,T}^{\mathrm{H}}(h_{\ell},h_{\ell'}))_{1 \leq \ell,\ell' \leq n}\right)\right]
= \begin{cases}
\psi(Q^{\mathrm{high}}), & \beta \leq \beta_{c}(\theta), \\
\EE_{\bb{\xi}}\big[\EE[\psi(Q^{\mathrm{low}}) \mid \bb{\xi}]\big], & \beta > \beta_{c}(\theta).
\end{cases}
\end{equation}
\end{proposition}

The sampling meaning of \eqref{eq:haar.gg} is as follows. Conditional on the Haar phases, sample $h_{1},\ldots,h_{n + 1}$ independently from $G_{\beta,\sig,T}^{\mathrm{H}}$ and fix $k \in \{1,\ldots,n\}$. After multiplication by any bounded deterministic function of the first $n$ positions and averaging over both the positions and the phases, the overlap $R_{\sig,T}^{\mathrm{H}}(h_{k},h_{n + 1})$ has asymptotically the same mean as a mixture assigning weight $1/n$ to an independent draw from the averaged two-overlap distribution and weight $1/n$ to each previously observed overlap $R_{\sig,T}^{\mathrm{H}}(h_{k},h_{j})$, $j \neq k$. The independent draw here is independent of the entire first sampling experiment, including its Haar phases.

At and below $\beta_{c}(\theta)$, every off-diagonal entry of the limiting overlap array equals $q_{0}$. Above this threshold, the sampled positions form groups in the limiting description: overlaps within a group equal $1$, and overlaps between distinct groups equal $q_{0}$. The two-overlap distribution gives the probability that two samples belong to the same group, while the GG identities determine how these groups are sampled jointly. Equation~\eqref{eq:haar.joint.array} describes the resulting groups through Poisson--Dirichlet weights of parameter $\beta_{c}(\theta)/\beta$.

The proof of Proposition~\ref{prop:haar.laws} is given in Section~\ref{sec:proofs}. Subsection~\ref{subsec:haar.identities} deduces its conclusions from four inputs: the free energy obtained by perturbing the Haar field by a multiple of its truncation $H_{\alpha}$, estimates for the covariance between $H_{\alpha}$ and $H$, an approximate integration by parts identity, and concentration of the normalized Haar field value $H(h)/t$ around a deterministic limit when $h$ is sampled from the Haar Gibbs measure. Subsection~\ref{subsec:PNT} collects the prime number theorem estimates used in their proofs, and the subsequent subsections establish these inputs. Figure~\ref{fig:haar.dependencies} records the dependencies between the main proposition and the supporting lemmas.

\vspace{3mm}
\begin{figure}[!ht]
\centering
\begin{tikzpicture}[x = 1cm, y = 1cm, every node/.style = {font = \small, align = center, execute at begin node = {\hyphenpenalty = 10000\exhyphenpenalty = 10000\relax}}, result/.style = {draw, rounded corners = 2pt, fill = black!15, text width = 3.8cm, minimum height = 1.1cm, inner sep = 4pt}, dependency/.style = {-{Stealth[length = 2mm]}, line width = 0.5pt}]
\node[result] (replacement) at (0,0) {Lemma~\ref{lem:replacement.estimates} \\Replacement, concentration, and supremum estimates};
\node[result] (comparison) at (5.1,0) {Lemma~\ref{lem:forest.comparison} \\Comparison with Gaussian forests};
\node[result] (forest) at (10.2,0) {Lemma~\ref{lem:forest.free.energy} \\Gaussian forest free energy};
\node[result] (thermal) at (0, - 2.4) {Lemma~\ref{lem:thermal.input} \\Energy concentration};
\node[result] (freeenergy) at (5.1, - 2.4) {Lemma~\ref{lem:free.energy.input} \\Free energy of the perturbed Haar field $H^{u}$};
\node[result] (covariance) at (10.2, - 2.4) {Lemma~\ref{lem:covariance.input} \\Truncated covariance geometry};
\node[result] (ibp) at (0, - 4.8) {Lemma~\ref{lem:haar.ibp} \\Approximate integration by parts};
\node[result] (laws) at (5.1, - 4.8) {Proposition~\ref{prop:haar.laws} \\Overlap distributions and approximate GG identities for $H$};
\draw[dependency] (replacement.east) -- (comparison.west);
\draw[dependency] (replacement.south east) -- (freeenergy.north west);
\draw[dependency] (replacement.south) -- (thermal.north);
\draw[dependency] (comparison.south) -- (freeenergy.north);
\draw[dependency] (forest.south west) -- (freeenergy.north east);
\draw[dependency] (freeenergy.west) -- (thermal.east);
\draw[dependency] (freeenergy.south) -- (laws.north);
\draw[dependency] (thermal.south east) -- (laws.north west);
\draw[dependency] (ibp.east) -- (laws.west);
\draw[dependency] (covariance.south west) -- (laws.north east);
\draw[dependency,dashed] (comparison.east) -- (forest.west);
\draw[dependency,dashed] (replacement.north) -- (0,1.5) -- (10.2,1.5) -- (forest.north);
\end{tikzpicture}
\caption{Proof dependencies for Proposition~\ref{prop:haar.laws} and the supporting lemmas of Section~\ref{sec:proofs}. Dashed arrows indicate that the proof of Lemma~\ref{lem:forest.free.energy} also uses the Gaussian concentration argument from the proof of Lemma~\ref{lem:replacement.estimates} and the treatment of the common Gaussian trunk in the proof of Lemma~\ref{lem:forest.comparison}.}\label{fig:haar.dependencies}
\end{figure}
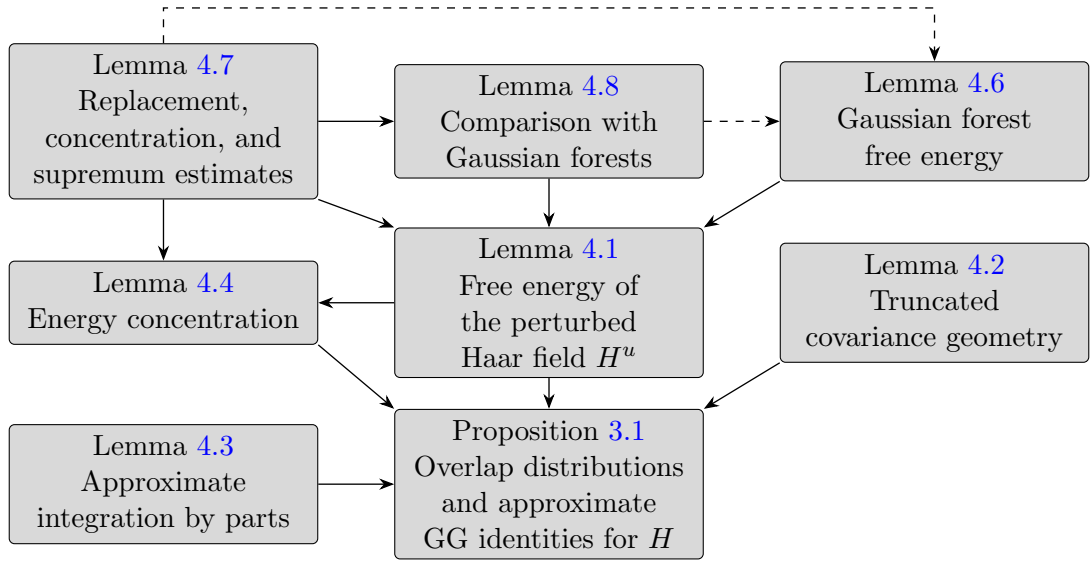

\section{Proofs}\label{sec:proofs}

\subsection{Deduction of the overlap distributions and approximate GG identities}\label{subsec:haar.identities}

The limiting two-overlap distribution of $R_{\sig,T}^{\mathrm{H}}(h_{1},h_{2})$ under $\EE_{\mathrm{H}}(G_{\beta,\sig,T}^{\mathrm{H}})^{\times 2}$ is determined by combining differentiation of the perturbed free energy $f_{T}(u)$ in \eqref{eq:perturbed.free.energy.definition} with the approximate integration by parts identity \eqref{eq:haar.ibp} of Lemma~\ref{lem:haar.ibp}. Lemma~\ref{lem:free.energy.input} gives the limiting free energy and its derivative. The differentiation is with respect to $u$ at $u = 0$, where $u$ multiplies the Haar field $H_{\alpha}$ truncated at prime scale $\alpha$. Applied to the full Haar field $H$, this approximate integration by parts identity also gives the approximate GG identities for bounded measurable functions of the sampled positions once $H(h) / t$ concentrates around a deterministic limit under the Haar Gibbs measure, as stated in Lemma~\ref{lem:thermal.input}. These two conclusions then determine the limiting joint distribution of the overlaps. The following lemmas isolate these steps before their proofs.

Recall the perturbed Haar field $H^{u} = H + uH_{\alpha}$ and its free energy $f_{T}(u)$ from \eqref{eq:perturbed.free.energy.definition}.

\begin{lemma}[Free energy of the perturbed Haar field $H^{u}$]\label{lem:free.energy.input}
For fixed $q_{0} < \alpha < 1$, the random functions $f_{T}(u)$ converge, as $T \to \infty$, pointwise in probability and in $L^{1}$, for $u$ in a neighborhood of zero, to a deterministic function $f(u)$ that is differentiable at zero. Its value at zero is the unperturbed limiting free energy $f_{\theta}(\beta)$, given by
\begin{equation}\label{eq:unperturbed.free.energy}
f_{\theta}(\beta)
=
\begin{cases}
\theta + \beta^{2}(1 - q_{0}) / 4, & \beta \leq \beta_{c}(\theta), \\
\beta(1 - q_{0})\beta_{c}(\theta) / 2 - 1, & \beta > \beta_{c}(\theta).
\end{cases}
\end{equation}
The function $f_{\theta}$ is differentiable at every $\beta > 0$, including $\beta_{c}(\theta)$. The left and right derivatives in $u$ at zero also agree:
\begin{equation}\label{eq:perturbation.derivative}
f'(0)
=
\begin{cases}
\beta^{2}(\alpha - q_{0}) / 2, & \beta \leq \beta_{c}(\theta), \\
\beta\beta_{c}(\theta)(\alpha - q_{0}) / 2, & \beta > \beta_{c}(\theta).
\end{cases}
\end{equation}
Moreover, $\EE_{\mathrm{H}}[f_{T}'(0)] \to f'(0)$.

\end{lemma}

\begin{lemma}[Truncated covariance geometry]\label{lem:covariance.input}
For every fixed $0 < \alpha \leq 1$, as $T \to \infty$, uniformly in $h,h' \in I_{\theta}$,
\begin{equation}\label{eq:truncated.overlap.covariance}
V_{T} = \frac{t}{2} + \OO(\log{t}),
\qquad
K_{\alpha}(h,h')
= \frac{t}{2}\min\{\alpha,\max\{R_{\sig,T}^{\mathrm{H}}(h,h'),0\}\}
 + \OO(\log{t}).
\end{equation}
In addition, $R_{\sig,T}^{\mathrm{H}}(h,h') = r_{T}(h,h') + \OO((\log{t}) / t)$ uniformly, where the deterministic logarithmic distance kernel is $r_{T}(h,h') = \max\{0,\min\{1,t^{-1}\log(1 / |h - h'|)\}\}$ for $h \neq h'$, with $r_{T}(h,h) = 1$. In particular, $R_{\sig,T}^{\mathrm{H}}(h,h') \geq q_{0} - \oo(1)$.
\end{lemma}

\begin{lemma}[Approximate integration by parts]\label{lem:haar.ibp}
For every $0 < \alpha \leq 1$, fixed $n \geq 1$, $1 \leq k \leq n$ and deterministic bounded measurable $\Psi_{T}:I_{\theta}^{n} \to \R$, we have, as $T \to \infty$,
\begin{equation}\label{eq:haar.ibp}
\begin{aligned}
&\EE_{\mathrm{H}} (G_{\beta,\sig,T}^{\mathrm{H}})^{\times n}[H_{\alpha}(h_{k})\Psi_{T}(\bb{h})] \\
&\qquad= \beta \sum_{j = 1}^{n}\EE_{\mathrm{H}} (G_{\beta,\sig,T}^{\mathrm{H}})^{\times n}[K_{\alpha}(h_{k},h_{j})\Psi_{T}(\bb{h})] \\
&\qquad\qquad - n\beta \EE_{\mathrm{H}} (G_{\beta,\sig,T}^{\mathrm{H}})^{\times(n+1)}[K_{\alpha}(h_{k},h_{n+1})\Psi_{T}(\bb{h})]
 + \OO(n^{2}\beta^{2}\|\Psi_{T}\|_{\infty}).
\end{aligned}
\end{equation}
The constant is independent of $T$ and $\Psi_{T}$.
\end{lemma}

\begin{lemma}[Energy concentration]\label{lem:thermal.input}
For every fixed $\beta > 0$, as $T \to \infty$,
\begin{equation}\label{eq:thermal.concentration}
\EE_{\mathrm{H}} G_{\beta,\sig,T}^{\mathrm{H}}
 \left[\left|\frac{H(h)}{t} - f_{\theta}'(\beta)\right|\right]
\to 0.
\end{equation}
\end{lemma}

\begin{proof}[Proof of Proposition~\ref{prop:haar.laws} from the four lemmas]
Differentiating \eqref{eq:perturbed.free.energy.definition} at $u = 0$ gives $f_{T}'(0) = \frac{\beta}{t}G_{\beta,\sig,T}^{\mathrm{H}}[H_{\alpha}(h)]$. Taking $n = 1$ and $\Psi_{T} \equiv 1$ in Lemma~\ref{lem:haar.ibp} therefore yields
\begin{equation}\label{eq:overlap.integrated.cdf}
\begin{aligned}
\frac{2}{\beta^{2}}\EE_{\mathrm{H}}[f_{T}'(0)]
&= \frac{2}{t}\EE_{\mathrm{H}} (G_{\beta,\sig,T}^{\mathrm{H}})^{\times 2}\left[K_{\alpha}(h_{1},h_{1}) - K_{\alpha}(h_{1},h_{2})\right] + \oo(1) \\
&= \EE_{\mathrm{H}} (G_{\beta,\sig,T}^{\mathrm{H}})^{\times 2}\left[\alpha - \min\{\alpha,\max\{R_{\sig,T}^{\mathrm{H}}(h_{1},h_{2}),0\}\}\right] + \oo(1) \\
&= \int_{0}^{\alpha}\EE_{\mathrm{H}} (G_{\beta,\sig,T}^{\mathrm{H}})^{\times 2} [\ind_{\{R_{\sig,T}^{\mathrm{H}}(h_{1},h_{2}) \leq y\}}] \rd{y} + \oo(1).
\end{aligned}
\end{equation}
The second equality uses Lemma~\ref{lem:covariance.input} and the diagonal identity $R_{\sig,T}^{\mathrm{H}}(h_{1},h_{1}) = 1$. Its uniform covariance errors contribute $\OO((\log{t}) / t)$ after division by $t$.

Let $\mu_{T}$ be the law of $R_{\sig,T}^{\mathrm{H}}(h_{1},h_{2})$ under the averaged Gibbs measure $\EE_{\mathrm{H}}(G_{\beta,\sig,T}^{\mathrm{H}})^{\times 2}$. It is supported in the compact interval $[ - 1,1]$, and the family $(\mu_{T})_{T}$ is tight. Consider any subsequence along which $\mu_{T}$ converges weakly to a probability measure $\mu$. Lemma~\ref{lem:covariance.input} gives a deterministic sequence $\e_{T} \to 0$ such that $R_{\sig,T}^{\mathrm{H}}(h,h') \geq q_{0} - \e_{T}$ uniformly in $h,h' \in I_{\theta}$. Consequently, for every $\e > 0$, $\mu_{T}(( - \infty,q_{0} - \e)) = 0$ for all sufficiently large $T$. The portmanteau theorem gives $\mu(( - \infty,q_{0} - \e)) = 0$. Letting $\e \downarrow 0$ shows that this subsequential limiting law $\mu$ is supported in $[q_{0},1]$.

For each fixed $q_{0} < \alpha < 1$, we take the limit in \eqref{eq:overlap.integrated.cdf}. Lemma~\ref{lem:free.energy.input} gives $\EE_{\mathrm{H}}[f_{T}'(0)] \to f'(0)$, with $f'(0)$ specified in \eqref{eq:perturbation.derivative}. Using the support of $\mu$ yields
\[
\int_{q_{0}}^{\alpha}\mu([q_{0},y])\rd{y}
= \int_{q_{0}}^{1}(\alpha - r)_{+}\mu(\rd{r})
= \frac{2f'(0)}{\beta^{2}}
= \min\left\{1, \frac{\beta_{c}(\theta)}{\beta}\right\}(\alpha - q_{0}).
\]
Taking right difference quotients and using right-continuity of the cumulative distribution function gives
\[
\mu([q_{0},\alpha]) = \min\{1,\beta_{c}(\theta) / \beta\}, \qquad q_{0} < \alpha < 1.
\]
Letting $\alpha \downarrow q_{0}$ identifies the mass at $q_{0}$, while letting $\alpha \uparrow 1$ identifies the mass below $1$. Hence
\[
\mu = \min\{1,\beta_{c}(\theta) / \beta\}\delta_{q_{0}} + \left(1 - \min\{1,\beta_{c}(\theta) / \beta\}\right)\delta_{1}.
\]
Every subsequential limiting law is therefore the same. Tightness then gives convergence of the full family $(\mu_{T})_{T}$, proving \eqref{eq:haar.two.overlap}.

The energy concentration estimate of Lemma~\ref{lem:thermal.input} also applies after multiplication by a bounded function of any fixed number of sampled positions. Indeed, the triangle inequality, the bound $|\Psi_{T}| \leq \|\Psi_{T}\|_{\infty}$, and the fact that the $k$th marginal of $(G_{\beta,\sig,T}^{\mathrm{H}})^{\times n}$ is $G_{\beta,\sig,T}^{\mathrm{H}}$ give
\[
\left|\EE_{\mathrm{H}} (G_{\beta,\sig,T}^{\mathrm{H}})^{\times n} \left[\left(\frac{H(h_{k})}{t} - f_{\theta}'(\beta)\right)\Psi_{T}(\bb{h})\right]\right|
\leq \|\Psi_{T}\|_{\infty}\EE_{\mathrm{H}} G_{\beta,\sig,T}^{\mathrm{H}} \left[\left|\frac{H(h)}{t} - f_{\theta}'(\beta)\right|\right].
\]
The right side tends to zero uniformly over $\|\Psi_{T}\|_{\infty} \leq 1$. Thus
\begin{equation}\label{eq:positional.energy}
\frac{1}{t}\EE_{\mathrm{H}} (G_{\beta,\sig,T}^{\mathrm{H}})^{\times n}[H(h_{k})\Psi_{T}(\bb{h})]
- f_{\theta}'(\beta)\EE_{\mathrm{H}} (G_{\beta,\sig,T}^{\mathrm{H}})^{\times n}[\Psi_{T}(\bb{h})] \to 0
\end{equation}
uniformly over that class of functions.

Apply Lemma~\ref{lem:haar.ibp} with $\alpha = 1$. Since $K_{1} = V_{T}R_{\sig,T}^{\mathrm{H}}$ and $V_{T}/t \to 1/2$ by Lemma~\ref{lem:covariance.input}, division by $\beta V_{T}$ and \eqref{eq:positional.energy} give
\begin{equation}\label{eq:sub.in.1}
\begin{aligned}
\frac{2f_{\theta}'(\beta)}{\beta}\EE_{\mathrm{H}} (G_{\beta,\sig,T}^{\mathrm{H}})^{\times n}[\Psi_{T}(\bb{h})]
&= \sum_{j = 1}^{n}\EE_{\mathrm{H}} (G_{\beta,\sig,T}^{\mathrm{H}})^{\times n}[R_{\sig,T}^{\mathrm{H}}(h_{k},h_{j})\Psi_{T}(\bb{h})] \\
&\qquad - n\EE_{\mathrm{H}} (G_{\beta,\sig,T}^{\mathrm{H}})^{\times(n+1)}[R_{\sig,T}^{\mathrm{H}}(h_{k},h_{n+1})\Psi_{T}(\bb{h})] + \oo(1).
\end{aligned}
\end{equation}
For fixed $n$ and $\beta$, the error remains uniform over $\|\Psi_{T}\|_{\infty} \leq 1$, because the remainder in Lemma~\ref{lem:haar.ibp}, divided by $\beta V_{T}$, tends to zero uniformly over this class. Taking $n = 1$ and $\Psi_{T} \equiv 1$ gives
\begin{equation}\label{eq:sub.in.2}
\frac{2f_{\theta}'(\beta)}{\beta}
= 1 - \EE_{\mathrm{H}} (G_{\beta,\sig,T}^{\mathrm{H}})^{\times 2}[R_{\sig,T}^{\mathrm{H}}(h_{1},h_{2})] + \oo(1).
\end{equation}
In the general identity, the summand with $j = k$ equals $\EE_{\mathrm{H}}(G_{\beta,\sig,T}^{\mathrm{H}})^{\times n}[\Psi_{T}(\bb{h})]$, since the diagonal overlap is $1$. Substituting \eqref{eq:sub.in.2} in \eqref{eq:sub.in.1} cancels this diagonal term, and reorganizing the terms yields
\[
\begin{aligned}
n\EE_{\mathrm{H}} (G_{\beta,\sig,T}^{\mathrm{H}})^{\times(n+1)}[R_{\sig,T}^{\mathrm{H}}(h_{k},h_{n+1})\Psi_{T}(\bb{h})]
&= \EE_{\mathrm{H}} (G_{\beta,\sig,T}^{\mathrm{H}})^{\times 2}[R_{\sig,T}^{\mathrm{H}}(h_{1},h_{2})]\EE_{\mathrm{H}} (G_{\beta,\sig,T}^{\mathrm{H}})^{\times n}[\Psi_{T}(\bb{h})] \\
&\quad+ \sum_{\substack{1 \leq j \leq n \\j \neq k}}\EE_{\mathrm{H}} (G_{\beta,\sig,T}^{\mathrm{H}})^{\times n}[R_{\sig,T}^{\mathrm{H}}(h_{k},h_{j})\Psi_{T}(\bb{h})] + \oo(1).
\end{aligned}
\]
Division by $n$ proves \eqref{eq:haar.gg}. When $n = 1$, the sum is empty, so the same argument applies. No continuity assumption on $\Psi_{T}$ has been used.

It remains to identify the limiting joint distribution of the overlaps. For each $T$, sample countably many positions conditionally independently from $G_{\beta,\sig,T}^{\mathrm{H}}$. Their overlap arrays take values in the compact countable product $[ - 1,1]^{\N\times\N}$. From any sequence of $T$ tending to infinity, successive subsequence extraction for the finite arrays therefore gives a consistent limiting array $(Q_{ij})_{i,j\geq 1}$. Its law is exchangeable, meaning invariant under simultaneous finite permutations of its row and column indices, because the sampled positions are conditionally independent with the same distribution. Each finite overlap matrix is a correlation matrix, so it is positive semidefinite and has diagonal $1$. These properties pass to the limit because the corresponding sets of matrices are closed.

By \eqref{eq:haar.two.overlap}, every off-diagonal entry of the limiting array belongs to $\{q_{0},1\}$ almost surely. This holds simultaneously for all pairs because there are countably many pairs. The entries equal to $1$ define an equivalence relation on the sampled indices. Reflexivity and symmetry are immediate. For transitivity, represent any three-by-three submatrix as the Gram matrix of unit vectors $\bb{v}_{i},\bb{v}_{j},\bb{v}_{k}$, meaning that its entries are their pairwise inner products. Then $\|\bb{v}_{i} - \bb{v}_{j}\|_{2}^{2} = 2(1 - Q_{ij})$, so $Q_{ij} = Q_{jk} = 1$ implies $\bb{v}_{i} = \bb{v}_{j} = \bb{v}_{k}$ and hence $Q_{ik} = 1$. Let $\Pi_{n}$ denote the resulting random partition of $\{1,\ldots,n\}$.

When $\beta \leq \beta_{c}(\theta)$, \eqref{eq:haar.two.overlap} implies that every off-diagonal entry equals $q_{0}$ almost surely. Thus every finite limiting array equals $(1 - q_{0})\mathrm{Id}_{n} + q_{0}\bb{1}\bb{1}^{\mathsf{T}}$, which proves the first case of \eqref{eq:haar.joint.array} by uniqueness of the subsequential limiting law.

Suppose now that $\beta > \beta_{c}(\theta)$, and set $\lambda = \beta_{c}(\theta) / \beta \in (0,1)$. The argument below reduces identification of the limiting overlap array to the classical Poisson--Dirichlet sampling formula, using the GG identities in the case of two-valued overlaps; see \citet[Proposition~15.2.4]{MR3024566} and \citet[Section~3, proof of Theorem~4]{Panchenko2011Representation}. Since $\theta > - 1$, $q_{0} < 1$. For a specified partition $\pi$ of $\{1,\ldots,n\}$, write $i \sim_{\pi} j$ when $i$ and $j$ belong to the same block, and define the polynomial
\[
P_{\pi}((r_{ij})_{i,j=1}^{n})
= \prod_{\substack{1 \leq i < j \leq n \\i \sim_{\pi} j}}\frac{r_{ij} - q_{0}}{1 - q_{0}}
\prod_{\substack{1 \leq i < j \leq n \\i \not\sim_{\pi} j}}\frac{1 - r_{ij}}{1 - q_{0}}.
\]
Empty products are interpreted as $1$. If $Q^{(n)} = (Q_{ij})_{i,j=1}^{n}$, the two-valued nature of the limiting overlaps gives
\[
P_{\pi}(Q^{(n)}) = \ind_{\{\Pi_{n} = \pi\}}.
\]
At finite $T$, the same polynomial evaluated at $(R_{\sig,T}^{\mathrm{H}}(h_{i},h_{j}))_{i,j=1}^{n}$ is a deterministic measurable function of the sampled positions. It is bounded uniformly in $T$, because the polynomial is bounded on the compact cube $[ - 1,1]^{n\times n}$. After division by a fixed bound, it is therefore an admissible function in \eqref{eq:haar.gg}. Polynomial continuity and convergence of the finite arrays permit passage to the limit, giving
\[
\EE[Q_{k,n+1}\ind_{\{\Pi_{n} = \pi\}}]
= \frac{1}{n}\EE[Q_{12}]\PP(\Pi_{n} = \pi) + \frac{1}{n}\sum_{\substack{1 \leq j \leq n \\j \neq k}}\EE[Q_{kj}\ind_{\{\Pi_{n} = \pi\}}].
\]

Fix a partition $\pi$ with positive probability and a block $B$ of $\pi$, and choose $k \in B$. On $\{\Pi_{n} = \pi\}$, exactly $|B| - 1$ terms in the sum have overlap $1$, and the other $n - |B|$ terms have overlap $q_{0}$. Moreover, \eqref{eq:haar.two.overlap} gives
\[
\EE[Q_{12}] = q_{0} + (1 - q_{0})(1 - \lambda).
\]
Dividing the limiting GG identity by $\PP(\Pi_{n} = \pi)$ therefore yields
\[
\EE[Q_{k,n+1} \mid \Pi_{n} = \pi]
= \frac{\EE[Q_{12}] + |B| - 1 + (n - |B|)q_{0}}{n}
= q_{0} + (1 - q_{0})\frac{|B| - \lambda}{n}.
\]
The overlap $Q_{k,n+1}$ equals $1$ precisely when sampled index $n + 1$ joins $B$, and otherwise equals $q_{0}$. Consequently,
\[
\PP(n + 1\text{ joins }B \mid \Pi_{n} = \pi) = \frac{|B| - \lambda}{n}.
\]
The events of joining the existing blocks are disjoint. Their complement is the event of starting a new block, whose conditional probability is
\[
1 - \sum_{B\in\pi}\frac{|B| - \lambda}{n} = \frac{\lambda|\pi|}{n}.
\]
Starting from the single block $\{1\}$, these rules uniquely determine the distribution of every finite partition.

For a specified partition with $k$ blocks of sizes $n_{1},\ldots,n_{k}$, reveal the sampled indices in increasing order. The denominators in the successive conditional probabilities contribute $\prod_{m=1}^{n-1}m = \Gamma(n)$. The creation of the second through the $k$th blocks contributes $\lambda^{k-1}\Gamma(k)$, while growth within block $i$ contributes
\[
\prod_{m=1}^{n_{i}-1}(m - \lambda) = \frac{\Gamma(n_{i} - \lambda)}{\Gamma(1 - \lambda)}.
\]
Thus the probability of this specified partition is
\begin{equation}\label{eq:haar.eppf}
p(n_{1},\ldots,n_{k})
= \frac{\lambda^{k-1}\Gamma(k)}{\Gamma(n)}
 \prod_{j = 1}^{k}\frac{\Gamma(n_{j} - \lambda)}{\Gamma(1 - \lambda)},
\qquad n = n_{1} + \cdots + n_{k}.
\end{equation}
There is no additional combinatorial factor because the partition itself, rather than only its block sizes, is specified.

By \citet[Proposition~51, Eq.~(178), p.~897]{MR1434129}, \eqref{eq:haar.eppf} is the sampling formula for Poisson--Dirichlet weights with parameter $\lambda$, denoted by $\mathrm{PD}(\lambda,0)$ in that reference. The limiting overlap array is determined by its partition: its entries equal $1$ within a block and $q_{0}$ between different blocks. Hence every subsequential limiting array has the law obtained by sampling conditionally independent group labels from these weights and assigning the two overlap values. Uniqueness of this limiting law, together with tightness, proves \eqref{eq:haar.joint.array} and completes the proof.
\end{proof}

\begin{remark}[Approximate ultrametricity of the Haar field]
Lemma~\ref{lem:covariance.input} also yields uniform approximate ultrametricity. Indeed, $|h_{1} - h_{2}| \leq 2\max\{|h_{1} - h_{3}|,|h_{2} - h_{3}|\}$ implies
\[
r_{T}(h_{1},h_{2})
\geq \min\{r_{T}(h_{1},h_{3}),r_{T}(h_{2},h_{3})\} - \frac{\log{2}}{t}.
\]
Since $R_{\sig,T}^{\mathrm{H}} = r_{T} + \OO((\log{t}) / t)$ uniformly, there exists $C > 0$ such that, for all sufficiently large $T$ and every $h_{1},h_{2},h_{3} \in I_{\theta}$,
\[
R_{\sig,T}^{\mathrm{H}}(h_{1},h_{2})
\geq \min\{R_{\sig,T}^{\mathrm{H}}(h_{1},h_{3}),R_{\sig,T}^{\mathrm{H}}(h_{2},h_{3})\} - C\frac{\log{t}}{t}.
\]
\end{remark}

\subsection{Prime number theorem estimates}\label{subsec:PNT}

The prime number theorem with its classical error term, given by \citet[Theorem~6.9, Eq.~(6.14), p.~179]{MR2378655}, states that
\begin{equation}\label{analytic:eq:PNT}
\pi(y) = \operatorname{Li}(y) + \OO\big(y e^{-c\sqrt{\log{y}}}\big),
\end{equation}
where $\pi(y)$ counts primes at most $y$, $\operatorname{Li}(y) = \int_{2}^{y}\rd{u}/\log{u}$ is the logarithmic integral, and $c > 0$ is an absolute constant. Another input is Mertens' estimate $\sum_{p\leq y}p^{-1} = \log{\log{y}} + \OO(1)$, as stated in \citet[Eq.~(3.14.5), p.~65]{MR882550}.

The covariance calculation, the Gaussian replacement and concentration estimates, and the Gaussian forest comparisons use the following consequences of the prime number theorem \eqref{analytic:eq:PNT}. The bounds are stated here so that their later applications only require references to the corresponding estimates. Throughout this subsection, $\theta > - 1$ is fixed, $(\sig,X)$ satisfies \eqref{eq:fixed.parameters}, and $\eta_{\sig}$ and $\Lambda_{T}$ are as in Subsection~\ref{subsec:parameters}. All estimates hold as $T \to \infty$, uniformly over this parameter range.

First, partial summation in \eqref{analytic:eq:PNT}, followed by the substitution $v = \log{y}$, gives
\begin{equation}\label{eq:pnt.weighted.sums}
\sum_{p \leq X}\frac{1}{p^{1 + 2\eta_{\sig}}} = \log{\Lambda_{T}} + \OO(1) = \OO(t),
\qquad
\sum_{p \leq X}\frac{(\log{p})^{2k}}{p^{1 + 2\eta_{\sig}}} = \OO_{k}(\Lambda_{T}^{2k}), \qquad k \geq 1,
\end{equation}
where $k$ is a fixed integer. The corresponding main integrals are $\smash{\int_{\log{2}}^{\log{X}}e^{-2\eta_{\sig}v}\rd{v}/v}$ and $\smash{\int_{\log{2}}^{\log{X}}v^{2k - 1}e^{-2\eta_{\sig}v}\rd{v}}$, respectively. Up to $\Lambda_{T}$, replacing $e^{-2\eta_{\sig}v}$ by $1$ in the first integral costs $\OO(\eta_{\sig}\Lambda_{T}) = \OO(1)$. If $\Lambda_{T} < \log{X}$, the remaining integral is $\OO(1)$ by exponential decay. The second integral is $\OO_{k}((\log{X})^{2k})$ by dropping the exponential, and, when $\eta_{\sig} > 0$, it is also $\OO_{k}(\eta_{\sig}^{-2k})$ by extending the upper endpoint to infinity. These estimates give the asserted bounds, including when $\eta_{\sig} = 0$. The prime number theorem remainder contributes $\OO_{k}(1)$ uniformly because $\eta_{\sig} = \oo(1)$ and every fixed power of $v$ is integrable against $e^{-c\sqrt{v}}\rd{v}$ on $[\log{2},\infty)$.

For later use with complete prime contributions, define the deterministic amplitude bound $b_{p}$ by
\begin{equation}\label{eq:prime.block.bound}
b_{p} = \sum_{j \geq 1}\frac{1}{j p^{j\sig}} \leq C p^{-\sig}, \qquad \sum_{p}b_{p}^{3} = \OO(1), \qquad \sum_{p \leq X}b_{p}^{2} = \OO(t).
\end{equation}
The first bound follows by summing a geometric series, uniformly for $\sig \geq 1/2$. The cubic sum is bounded by $C\sum_{n \geq 2}n^{-3/2}$, and the quadratic sum follows from \eqref{eq:pnt.weighted.sums}.

Higher prime powers contribute a bounded amount to covariance sums. Their differentiated amplitudes also have summable squares: for every fixed integer $k \geq 1$,
\begin{equation}\label{eq:pnt.higher.powers}
\begin{aligned}
\sum_{p}\sum_{j \geq 2}\frac{1}{j^{2}p^{2j\sig}}
&\leq \sum_{p} \frac{1}{p^2} \left(1 - \frac{1}{p}\right)^{-1} = \sum_{p}\frac{1}{p(p - 1)} < \infty, \\
\sum_{p}\left(\sum_{j \geq 2}\frac{j^{k - 1}(\log{p})^{k}}{p^{j\sig}}\right)^{2}
&\leq C_{k}\sum_{p}\frac{(\log{p})^{2k}}{p^{4\sig}} = \OO_{k}(1).
\end{aligned}
\end{equation}
Indeed, the inner sum in the second line is at most $C_{k}(\log{p})^{k}p^{-2\sig}$, since $p^{-\sig} \leq 2^{-1/2}$. The last prime sum is bounded by the corresponding sum over integers with denominator $n^{2}$.

The spatial covariance estimates require an oscillatory version of \eqref{eq:pnt.weighted.sums}. For every fixed $0 < \alpha \leq 1$, uniformly for $h,h' \in I_{\theta}$,
\begin{equation}\label{eq:pnt.truncated.cosine}
\sum_{p \leq \exp((\log{X})^{\alpha})}\frac{\cos((h - h')\log{p})}{p^{1 + 2\eta_{\sig}}}
= t\min\{\alpha,r_{T}(h,h')\} + \OO(\log{t}),
\end{equation}
where $r_{T}$ is the logarithmic distance kernel in Lemma~\ref{lem:covariance.input}, with $r_{T}(h,h) = 1$.

To justify uniformity for $h,h' \in I_{\theta}$, where $I_{\theta}$ expands as $T \to \infty$ when $\theta > 0$, let $d = h - h'$, $v_{0} = Kt^{2}$ and $v_{1} = (\log{X})^{\alpha}$, where $K > 0$ is a sufficiently large fixed constant. Since $\alpha > 0$ is fixed, $v_{1} > v_{0}$ for all sufficiently large $T$. Mertens' estimate, recalled after \eqref{analytic:eq:PNT}, bounds the contribution of primes with $\log{p} \leq v_{0}$ by $\OO(\log{v_{0}}) = \OO(\log{t})$, uniformly in $d$. On the remaining range, partial summation in \eqref{analytic:eq:PNT} gives
\begin{equation}\label{eq:pnt.oscillatory.integral}
\sum_{e^{v_{0}} < p \leq e^{v_{1}}}\frac{\cos(d\log{p})}{p^{1 + 2\eta_{\sig}}}
= \int_{v_{0}}^{v_{1}}\frac{e^{-2\eta_{\sig}v}\cos(dv)}{v}\rd{v} + \OO\left((1 + |d|)(1 + \sqrt{v_{0}})e^{-c\sqrt{v_{0}}}\right).
\end{equation}
Here $c > 0$ may be smaller than the constant in \eqref{analytic:eq:PNT}. The factor $1 + |d|$ comes from differentiating the weight $y^{-1 - 2\eta_{\sig}}\cos(d\log{y})$. Since $|d| \leq 2L_{T} = 2e^{\theta t}$, choosing $K$ so that $c\sqrt{K} > \theta \vee 0$ makes this remainder $\oo(1)$, uniformly in $h,h'$.

Set $M = \min\{v_{1},\eta_{\sig}^{-1},|d|^{-1}\}$, with $\eta_{\sig}^{-1} = \infty$ when $\eta_{\sig} = 0$ and $|d|^{-1} = \infty$ when $d = 0$. Then
\begin{equation}\label{eq:pnt.oscillatory.integral.2}
\int_{v_{0}}^{v_{1}}\frac{e^{-2\eta_{\sig}v}\cos(dv)}{v}\rd{v} = \logp\left(\frac{M}{v_{0}}\right) + \OO(1).
\end{equation}
If $M \geq v_{0}$, replacing $e^{-2\eta_{\sig}v}\cos(dv)$ by $1$ on $[v_{0},M]$ costs $\OO(\eta_{\sig}M + d^{2}M^{2}) = \OO(1)$. The remaining integral is bounded by exponential decay if $M = \eta_{\sig}^{-1} < \infty$, or by integration by parts using the decreasing function $e^{-2\eta_{\sig}v}/v$ if $M = |d|^{-1}$; it is empty if $M = v_{1}$. The same bounds apply to the entire integral when $M < v_{0}$. Finally, \eqref{eq:effective.cutoff.bounds} gives $\log{v_{1}} = \alpha\log{\log{X}} = \alpha t + \OO(\log{t})$. Since $\eta_{\sig}^{-1} \geq \Lambda_{T}$, with the convention at zero stated above, it also gives
\[
\log\left(\min\{v_{1},\eta_{\sig}^{-1}\}\right) = \alpha t + \OO(\log{t}), \qquad \log{v_{0}} = \OO(\log{t}).
\]
Indeed, the first logarithm is at most $\log{v_{1}}$ and at least $\min\{\log{v_{1}},\log{\Lambda_{T}}\}$, and both bounds equal $\alpha t + \OO(\log{t})$ for fixed $0 < \alpha \leq 1$. These estimates give $\logp(M / v_{0}) = t\min\{\alpha,r_{T}(h,h')\} + \OO(\log{t})$. Combining this with \eqref{eq:pnt.oscillatory.integral}, \eqref{eq:pnt.oscillatory.integral.2}, and the bound for the removed primes proves \eqref{eq:pnt.truncated.cosine}.

For the perturbed Haar field $H^{u} = H + uH_{\alpha}$, the same estimates give the covariance profile used in the proof of Lemma~\ref{lem:forest.comparison}. Orthogonality of the powers of each Haar phase and the definition of $K_{\alpha}$ in \eqref{eq:haar.covariance.definitions} give $\Cov_{\mathrm{H}}(H_{\alpha}(h),H_{\alpha}(h')) = K_{\alpha}(h,h')$, so
\[
\Cov_{\mathrm{H}}(H^{u}(h),H^{u}(h')) = K_{1}(h,h') + (2u + u^{2})K_{\alpha}(h,h').
\]
Applying \eqref{eq:pnt.truncated.cosine} at scales $1$ and $\alpha$, and using \eqref{eq:pnt.higher.powers} for the bounded contribution of higher prime powers, yields, uniformly for $h,h' \in I_{\theta}$ and $u \in [ - 1/2,1/2]$,
\begin{equation}\label{eq:perturbed.covariance}
\begin{aligned}
\Cov_{\mathrm{H}}(H^{u}(h),H^{u}(h'))
&= \frac{t}{2}A_{u}(r_{T}(h,h')) + \OO(\log{t}), \\
A_{u}(r) &= (1 + u)^{2}\min\{r,\alpha\} + (r - \alpha)_{+}.
\end{aligned}
\end{equation}
The function $A_{u}$ is the cumulative variance profile expressed in units of $t / 2$; it is nondecreasing and Lipschitz. Equations~\eqref{eq:pnt.truncated.cosine} and \eqref{eq:pnt.higher.powers} will also prove the truncated covariance geometry in Lemma~\ref{lem:covariance.input}, while \eqref{eq:pnt.weighted.sums} and \eqref{eq:prime.block.bound} supply the prime sums needed for concentration, Gaussian replacement and approximate integration by parts.

\subsection{Energy concentration from differentiability of the free energy}\label{subsec:thermal.concentration}

\begin{proof}[Proof of Lemma~\ref{lem:thermal.input}]
Given $\e > 0$, the differentiability of $f_{\theta}$ in \eqref{eq:unperturbed.free.energy}, supplied by Lemma~\ref{lem:free.energy.input}, allows a $z > 0$ such that
\[
f_{\theta}(\beta + z) - f_{\theta}(\beta)
< z(f_{\theta}'(\beta) + \e / 2).
\]
Writing $\smash{f_{\theta,T}(b) = t^{-1}\log\int_{I_{\theta}} e^{bH(h)} \rd{h}}$ for the free energy as a function of inverse temperature in this proof, the exponential Markov bound within the Gibbs measure $G_{\beta,\sig,T}^{\mathrm{H}}$ gives
\[
\begin{aligned}
G_{\beta,\sig,T}^{\mathrm{H}} \left(\frac{H(h)}{t} \geq f_{\theta}'(\beta) + \e\right)
&\leq \exp\left( - tz(f_{\theta}'(\beta) + \e)\right)G_{\beta,\sig,T}^{\mathrm{H}}[\exp(zH(h))] \\
&= \exp\left( - tz(f_{\theta}'(\beta) + \e)\right)\frac{\exp(tf_{\theta,T}(\beta + z))}{\exp(tf_{\theta,T}(\beta))} \\
&= \exp\left(t\left[f_{\theta,T}(\beta + z) - f_{\theta,T}(\beta) - z(f_{\theta}'(\beta) + \e)\right]\right).
\end{aligned}
\]
By Lemma~\ref{lem:free.energy.input}, free energy convergence makes the right side exponentially small on an event whose probability tends to one.

For the lower tail, differentiability allows a $0 < z < \beta$ such that
\[
f_{\theta}(\beta - z) - f_{\theta}(\beta)
< - z(f_{\theta}'(\beta) - \e / 2).
\]
The same exponential Markov argument applied to $\exp( - zH(h))$ gives
\[
\begin{aligned}
G_{\beta,\sig,T}^{\mathrm{H}} \left(\frac{H(h)}{t} \leq f_{\theta}'(\beta) - \e\right)
&\leq \exp\left(tz(f_{\theta}'(\beta) - \e)\right)G_{\beta,\sig,T}^{\mathrm{H}}[\exp( - zH(h))] \\
&= \exp\left(tz(f_{\theta}'(\beta) - \e)\right)\frac{\exp(tf_{\theta,T}(\beta - z))}{\exp(tf_{\theta,T}(\beta))} \\
&= \exp\left(t\left[f_{\theta,T}(\beta - z) - f_{\theta,T}(\beta) + z(f_{\theta}'(\beta) - \e)\right]\right).
\end{aligned}
\]
Again, Lemma~\ref{lem:free.energy.input} makes the right side exponentially small on an event whose probability tends to one. Since Gibbs probabilities are bounded by $1$, the two tail bounds imply
\[
\EE_{\mathrm{H}} G_{\beta,\sig,T}^{\mathrm{H}} \left(\left|\frac{H(h)}{t} - f_{\theta}'(\beta)\right| \geq \e\right) \to 0.
\]
At $u = 0$, the first bound in \eqref{eq:field.supremum} of Lemma~\ref{lem:replacement.estimates} gives
\[
\EE_{\mathrm{H}}G_{\beta,\sig,T}^{\mathrm{H}}\left[\left(\frac{H(h)}{t}\right)^{2}\right]
\leq \frac{1}{t^{2}}\EE_{\mathrm{H}}\left[\sup_{h \in I_{\theta}}|H(h)|^{2}\right]
\leq C.
\]
Thus the sampled normalized energies are uniformly integrable under the averaged Gibbs measures. Together with the preceding tail convergence, this proves the $L^{1}$ convergence in \eqref{eq:thermal.concentration}.
\end{proof}

\subsection{Evaluation of the perturbed free energy}\label{subsec:free.energy}

A Gaussian forest is a collection of rooted trees with independent centered Gaussian increments on their edges; the value of the associated Gaussian field at a leaf is the sum of the increments along its path from a root. The trees carry time-inhomogeneous branching random walks: time corresponds to the logarithmic scale, and the variance of an increment per unit scale may depend on the stage. In the two-stage model used here, this variance rate is constant within each stage and changes at the boundary between the stages. Different leaves share precisely the increments on their common ancestral path, and different trees are independent before any spatially constant common Gaussian contribution is added. The trees represent the spatial intervals on the right panel in Figure~\ref{fig:prime.scale.regimes}; their precise subdivision construction is given in the proof of Lemma~\ref{lem:forest.comparison}.

The free energy input $f_{T}(u) \to f(u)$ follows from the comparison bounds in Lemma~\ref{lem:forest.comparison}, the two-stage Gaussian forest formula in Lemma~\ref{lem:forest.free.energy}, and the concentration estimate in Lemma~\ref{lem:replacement.estimates}. For every fixed subdivision factor $b$, the comparison lemma bounds the expected free energy $\EE_{\mathrm{H}}[f_{T}(u)]$ between those of the retained and full forests, up to errors tending to zero as $T \to \infty$. Their branching entropies are $\rho_{b} = \log(b / 2) / \log{b}$ and $1$, respectively. After taking $T \to \infty$, letting $b \to \infty$ makes the two forest limits agree. It then suffices to maximize the variational expression for the forest with branching entropy $1$ over its two stage energy rates. The replacement and concentration estimates are proved in Subsection~\ref{subsec:replacement}, and the comparison lemma and forest free energy formula are proved in Subsection~\ref{subsec:forests}.

Let $A = \alpha - q_{0}$ and $B = 1 - \alpha$ be the lengths of the two scale ranges after division by $t$.

\begin{lemma}[Gaussian forest free energy]\label{lem:forest.free.energy}
Fix $b > 1$ and an integer $m \geq 2$. Consider a Gaussian forest in which every nonterminal vertex has exactly $m$ children and every generation advances the logarithmic scale by $\log{b}$, with $\exp((\theta \vee 0)t + \OO(1))$ initial roots, stage lengths $At$ and $Bt$ rounded to neighboring generations, branching entropy $\rho = \log{m} / \log{b}$, meaning the logarithmic growth rate of the number of descendants per unit scale length, and variance rates $(1 + u)^{2} / 2$ and $1 / 2$, meaning the increment variances per unit scale length. Write $\mathcal{L}_{t}$ for its set of leaves and $X(h)$ for the sum of the Gaussian increments along the path to leaf $h$. Its free energy is $t^{-1}\log(e^{-t}\sum_{h \in \mathcal{L}_{t}}e^{\beta X(h)})$, where the weight $e^{-t}$ per leaf represents the length of a terminal spatial interval in the approximation of the Lebesgue integral. The limit of this free energy equals
\begin{equation}\label{eq:forest.variational}
- 1 + (\theta \vee 0) + \rho(A + B)
 + \sup_{(v_{1},v_{2})\in\mathcal{D}_{\rho}}
 \left\{\beta(Av_{1} + Bv_{2}) - \frac{Av_{1}^{2}}{(1 + u)^{2}} - Bv_{2}^{2}\right\},
\end{equation}
where $v_{1}$ and $v_{2}$ are the energy increments per unit scale length in the two stages, called the stage energy rates. The set $\mathcal{D}_{\rho}$ consists of the nonnegative pairs for which the Gaussian deviation rate does not exceed the logarithmic growth rate of the number of available paths at the end of either stage:
\begin{equation}\label{eq:forest.feasible}
\mathcal{D}_{\rho}
= \left\{(v_{1},v_{2})\in[0,\infty)^{2}:
\frac{Av_{1}^{2}}{(1 + u)^{2}} \leq (\theta \vee 0) + \rho A, \frac{Av_{1}^{2}}{(1 + u)^{2}} + Bv_{2}^{2} \leq (\theta \vee 0) + \rho(A + B)\right\}.
\end{equation}
The first inequality prevents the expected number of paths with the prescribed first-stage increment from being exponentially small, and the second imposes the same requirement at the leaves after both stages. A pair is called feasible when it satisfies both inequalities, and strictly feasible when both inequalities are strict. Convergence of the free energy holds in probability and in $L^{1}$ as $t \to \infty$.
\end{lemma}

\begin{proof}[Proof of Lemma~\ref{lem:free.energy.input}]
Recall that $q_{0} = |\theta \wedge 0|$, $A = \alpha - q_{0} > 0$, and $B = 1 - \alpha > 0$. In particular, $A + B = 1 - q_{0}$ and $(\theta \vee 0) + A + B = 1 + \theta$. For the local evaluation below, $u$ is further restricted to a sufficiently small neighborhood of zero within $[ - 1/2,1/2]$. Write
\[
\mathcal{J}_{u}(v_{1},v_{2}) = \beta(Av_{1} + Bv_{2}) - \frac{Av_{1}^{2}}{(1 + u)^{2}} - Bv_{2}^{2}
\]
for the objective in \eqref{eq:forest.variational}, and define the deterministic forest limit
\[
\mathcal{F}_{\rho}(u) = - 1 + (\theta \vee 0) + \rho(A + B) + \sup_{(v_{1},v_{2})\in\mathcal{D}_{\rho}}\mathcal{J}_{u}(v_{1},v_{2}), \qquad 0 < \rho \leq 1.
\]
Since $A,B > 0$, the feasible set $\mathcal{D}_{\rho}$ is compact and the objective is strictly concave.

For every fixed sufficiently large even subdivision factor $b$, Lemma~\ref{lem:forest.comparison} bounds $\EE_{\mathrm{H}}[f_{T}(u)]$ between the expected free energies of the retained and full forests, up to errors tending to zero as $T \to \infty$. The full forest has branching entropy $1$, whereas the retained forest has branching entropy $\rho_{b} = \log(b / 2) / \log{b}$. Applying the $L^{1}$ convergence in Lemma~\ref{lem:forest.free.energy} to the bounds \eqref{eq:forest.comparison} therefore gives, for every fixed admissible $u$,
\[
\mathcal{F}_{\rho_{b}}(u)
\leq \liminf_{T\to\infty}\EE_{\mathrm{H}}[f_{T}(u)]
\leq \limsup_{T\to\infty}\EE_{\mathrm{H}}[f_{T}(u)]
\leq \mathcal{F}_{1}(u).
\]
Here $T \to \infty$ also gives $t = \log\log T \to \infty$, as required when applying the forest lemma.

To pass to $b \to \infty$, observe that $\mathcal{D}_{\rho} \subseteq \mathcal{D}_{1}$ for $\rho \leq 1$. Conversely, if $(v_{1},v_{2}) \in \mathcal{D}_{1}$, then $(\sqrt{\rho}v_{1},\sqrt{\rho}v_{2}) \in \mathcal{D}_{\rho}$: both quadratic costs are multiplied by $\rho$, and $\rho(\theta \vee 0) \leq \theta \vee 0$. Applying this observation to a maximizing pair in $\mathcal{D}_{1}$ and using continuity of $\mathcal{J}_{u}$ shows that $\mathcal{F}_{\rho}(u) \to \mathcal{F}_{1}(u)$ as $\rho \uparrow 1$. Since $\rho_{b} \uparrow 1$, the preceding bounds imply
\[
g_{T}(u) \leqdef \EE_{\mathrm{H}}[f_{T}(u)] \to f(u) \leqdef \mathcal{F}_{1}(u) = \theta + \sup_{(v_{1},v_{2})\in\mathcal{D}_{1}}\mathcal{J}_{u}(v_{1},v_{2})
\]
as $T \to \infty$.

The concentration estimate \eqref{eq:log.partition.function.concentration} in Lemma~\ref{lem:replacement.estimates}, applied with deviation threshold $ty$, gives
\[
\PP_{\mathrm{H}}(|f_{T}(u) - g_{T}(u)| > y) \leq 2\exp( - cty^{2}), \qquad y > 0.
\]
Integrating this bound yields
\[
\EE_{\mathrm{H}}[|f_{T}(u) - g_{T}(u)|]
\leq 2\int_{0}^{\infty}\exp( - cty^{2})\rd{y}
\leq \frac{C}{\sqrt{t}}.
\]
Consequently, $f_{T}(u) \to f(u)$ in $L^{1}$ and in probability as $T \to \infty$. It remains to evaluate the deterministic supremum and its derivative at zero.

Set $w_{1} = v_{1}/(1 + u)$ and $w_{2} = v_{2}$. The feasible domain becomes the fixed compact set
\[
\mathcal{D}
= \left\{(w_{1},w_{2})\in[0,\infty)^{2}:Aw_{1}^{2} \leq (\theta \vee 0) + A, Aw_{1}^{2} + Bw_{2}^{2} \leq 1 + \theta\right\},
\]
and
\[
f(u) = \theta + \max_{(w_{1},w_{2})\in\mathcal{D}}\left\{\beta\big(A(1 + u)w_{1} + Bw_{2}\big) - Aw_{1}^{2} - Bw_{2}^{2}\right\}.
\]
At $u = 0$, Cauchy-Schwarz gives $Aw_{1} + Bw_{2} \leq \sqrt{(A + B)(Aw_{1}^{2} + Bw_{2}^{2})}$, with equality precisely when $w_{1} = w_{2}$. The maximum subject only to the second constraint is therefore attained uniquely at $w_{1} = w_{2} = c$, where
\[
c = \min\left\{\frac{\beta}{2},\sqrt{\frac{1 + \theta}{A + B}}\right\} = \frac{1}{2}\min\{\beta,\beta_{c}(\theta)\}.
\]
This pair satisfies the first constraint: if $\theta \leq 0$, then $A + B = 1 + \theta$ and $c \leq 1$, while if $\theta > 0$, then $A + B = 1$ and $Ac^{2} \leq A(1 + \theta) \leq A + \theta$, since $A < 1$. Consequently, $f(0) = \theta + (A + B)(\beta c - c^{2})$, which is \eqref{eq:unperturbed.free.energy}. Its two branches have the same derivative at $\beta = \beta_{c}(\theta)$, so
\[
f_{\theta}'(\beta) = \frac{1 - q_{0}}{2}\min\{\beta,\beta_{c}(\theta)\}.
\]
The objective is strictly concave in $(w_{1},w_{2})$, and its derivative with respect to $u$ is continuous on a neighborhood of $\{0\}\times\mathcal{D}$. The envelope theorem of \citet[Corollary~4(iii), p.~592]{MilgromSegal2002} thus gives differentiability at zero, including at the critical inverse temperature, and
\[
f'(0) = \beta Ac = \frac{\beta(\alpha - q_{0})}{2}\min\{\beta,\beta_{c}(\theta)\},
\]
which proves \eqref{eq:perturbation.derivative}.

It remains to prove convergence of the expected derivatives. For each fixed $T$, differentiation of the defining integral gives
\[
|f_{T}'(u)| \leq \frac{\beta}{t}\sup_{h\in I_{\theta}}|H_{\alpha}(h)|.
\]
The right side is bounded by a finite deterministic constant for fixed $T$, since $H_{\alpha}$ is a finite prime-power sum. Differentiation can therefore be interchanged with expectation, giving $g_{T}'(u) = \EE_{\mathrm{H}}[f_{T}'(u)]$. Moreover, the second derivative of $f_{T}$ is $\beta^{2}/t$ times the variance of $H_{\alpha}(h)$ when $h$ is sampled from the Gibbs measure of $H^{u}$. Thus $f_{T}$ and $g_{T}$ are convex.

For every sufficiently small fixed $s > 0$, convexity gives
\[
\frac{g_{T}(0) - g_{T}( - s)}{s}
\leq g_{T}'(0)
\leq \frac{g_{T}(s) - g_{T}(0)}{s}.
\]
The pointwise convergence $g_{T} \to f$ proved above consequently implies
\[
\frac{f(0) - f( - s)}{s}
\leq \liminf_{T\to\infty}\EE_{\mathrm{H}}[f_{T}'(0)]
\leq \limsup_{T\to\infty}\EE_{\mathrm{H}}[f_{T}'(0)]
\leq \frac{f(s) - f(0)}{s}.
\]
Letting $s \downarrow 0$ and using differentiability of $f$ at zero proves $\EE_{\mathrm{H}}[f_{T}'(0)] \to f'(0)$.
\end{proof}

\subsection{The covariance and approximate integration by parts inputs}\label{subsec:integration}

Let $B_{p}(h) = \Re\sum_{j: p^{j} \leq X}U_{p}^{j}/(j p^{j(\sig + \ii h)})$ be the complete contribution of $p$ to $H$, let $B_{p,\alpha}$ be its contribution to $H_{\alpha}$, and define its cross-covariance kernel $K_{p,\alpha}(h,h') = \EE_{\mathrm{H}}[B_{p,\alpha}(h)B_{p}(h')]$ and recall from \eqref{eq:prime.block.bound} the deterministic bound $\smash{b_{p} = \sum_{j \geq 1}1 / (j p^{j\sig})}$ on the absolute value of either prime contribution. The powers $U_{p}^{j}$ are the angular harmonics of $U_{p}$, and their orthogonality means that $\smash{\EE_{\mathrm{H}}[U_{p}^{j}\overline{U_{p}^{k}}]}$ equals $1$ for $j = k$ and $0$ otherwise. This gives the exact identity
\begin{equation}\label{eq:prime.covariance}
K_{p,\alpha}(h,h')
= \frac{1}{2} \sum_{\substack{j \geq 1 \\p^{j} \leq \exp((\log{X})^{\alpha})}} \frac{\cos(j(h - h')\log{p})}{j^{2}p^{2j\sig}}.
\end{equation}

The approximate integration by parts identity \eqref{eq:haar.ibp} in Lemma~\ref{lem:haar.ibp} is proved by temporarily removing one complete prime contribution $B_{p}$ from $H$. The resulting field $H^{(-p)} = H - B_{p}$, called the cavity field, depends only on the phases $U_{q}$ with $q \neq p$ and is therefore independent of $U_{p}$. All powers of $p$ are removed together because they are functions of that same phase. Restoring $B_{p}$ by interpolation and Taylor expansion produces the covariance terms in the integration by parts identity and a remainder bounded by $Cn^{2}\beta^{2}\|\Psi_{T}\|_{\infty}b_{p}^{3}$ for each prime. This removal and restoration procedure is called the cavity method. The exact covariance \eqref{eq:prime.covariance} and the cubic summability bound \eqref{eq:prime.block.bound} from Subsection~\ref{subsec:PNT} are used in the proof below.

\begin{proof}[Proof of Lemma~\ref{lem:haar.ibp}]
For the cavity field $H^{(-p)}$ defined above, let $H_{z} = H^{(-p)} + zB_{p}$, $0 \leq z \leq 1$, be the interpolating field, and let $G_{z}$ be its Gibbs measure, with density proportional to $e^{\beta H_{z}(h)}$. At $z = 0$, the prime contribution is absent, so $G_{0}$ is the Gibbs measure of the cavity field and is independent of $U_{p}$. At $z = 1$, the contribution is fully restored, so $H_{1} = H$ and $G_{1} = G_{\beta,\sig,T}^{\mathrm{H}}$. Define the partially averaged Gibbs expectation
\[
F_{p}(z) = \EE_{U_{p}}G_{z}^{\times n}[B_{p,\alpha}(h_{k})\Psi_{T}(\bb{h})],
\]
where $\EE_{U_{p}}$ denotes expectation over $U_{p}$ alone, with the other prime phases held fixed. Since $G_{0}$ is independent of $U_{p}$ and $\EE_{U_{p}}[B_{p,\alpha}(h)] = 0$ for every fixed $h$, averaging over $U_{p}$ inside the Gibbs integral gives $F_{p}(0) = 0$. Differentiating the Gibbs density at $z = 0$ inserts one factor $B_{p}$; its expectation with $B_{p,\alpha}$ is the exact covariance $K_{p,\alpha}$ from \eqref{eq:prime.covariance}. This gives
\[
F_{p}'(0) = \beta \sum_{j = 1}^{n} (G_{0})^{\times n}[K_{p,\alpha}(h_{k},h_{j})\Psi_{T}(\bb{h})] - n\beta (G_{0})^{\times(n+1)} [K_{p,\alpha}(h_{k},h_{n+1})\Psi_{T}(\bb{h})].
\]
The second derivative contains three factors from the same prime of the form $B_{p}(h_{j})$ or $B_{p,\alpha}(h_{k})$ and at most $Cn^{2}$ Gibbs terms, so
\[
\sup_{0 \leq z \leq 1}|F_{p}''(z)|
\leq Cn^{2}\beta^{2}\|\Psi_{T}\|_{\infty}b_{p}^{3}.
\]
To replace the covariance expectations under $G_{0}$ by their counterparts under $G_{1}$, let $\widetilde{\Psi}$ be a bounded deterministic function of $r$ sampled positions. Differentiating the Gibbs density and integrating over $z \in [0,1]$ gives
\[
\begin{aligned}
&\EE_{U_{p}}G_{1}^{\times r}[\widetilde{\Psi}(h_{1},\ldots,h_{r})] - G_{0}^{\times r}[\widetilde{\Psi}(h_{1},\ldots,h_{r})] \\
&\qquad= \beta\int_{0}^{1}\EE_{U_{p}}G_{z}^{\times(r + 1)}\left[\widetilde{\Psi}(h_{1},\ldots,h_{r})\left(\sum_{j = 1}^{r}B_{p}(h_{j}) - rB_{p}(h_{r + 1})\right)\right]\rd{z}.
\end{aligned}
\]
The term at $z = 0$ needs no expectation over $U_{p}$ because $G_{0}$ and $\widetilde{\Psi}$ are independent of $U_{p}$. Since $|B_{p}(h)| \leq b_{p}$, the parenthesized factor is bounded in absolute value by $2rb_{p}$. Consequently,
\[
\left|\EE_{U_{p}}G_{1}^{\times r}[\widetilde{\Psi}(h_{1},\ldots,h_{r})] - G_{0}^{\times r}[\widetilde{\Psi}(h_{1},\ldots,h_{r})]\right|
\leq 2r\beta b_{p}\|\widetilde{\Psi}\|_{\infty}.
\]
Moreover, the definition of $K_{p,\alpha}$ gives
\[
|K_{p,\alpha}(h,h')|
= \left|\EE_{U_{p}}[B_{p,\alpha}(h)B_{p}(h')]\right|
\leq b_{p}^{2}.
\]
Apply the preceding comparison with $r = n$ and $\widetilde{\Psi}(h_{1},\ldots,h_{n}) = K_{p,\alpha}(h_{k},h_{j})\Psi_{T}(\bb{h})$ for each $j \in \{1,\ldots,n\}$, and with $r = n + 1$ and $\widetilde{\Psi}(h_{1},\ldots,h_{n + 1}) = K_{p,\alpha}(h_{k},h_{n + 1})\Psi_{T}(\bb{h})$ for the last term in $F_{p}'(0)$. Each such function has supremum norm at most $b_{p}^{2}\|\Psi_{T}\|_{\infty}$. Accounting for the $n$ coefficients $\beta$ and the last coefficient $n\beta$, replacing all covariance expectations in $F_{p}'(0)$ by their counterparts under $G_{1}$ gives
\[
\begin{aligned}
&\left|F_{p}'(0) - \left\{\beta\sum_{j = 1}^{n}\EE_{U_{p}}G_{1}^{\times n}[K_{p,\alpha}(h_{k},h_{j})\Psi_{T}(\bb{h})] - n\beta\EE_{U_{p}}G_{1}^{\times(n + 1)}[K_{p,\alpha}(h_{k},h_{n + 1})\Psi_{T}(\bb{h})]\right\}\right| \\
&\qquad\leq \left(2n^{2} + 2n(n + 1)\right)\beta^{2}\|\Psi_{T}\|_{\infty}b_{p}^{3}
\leq Cn^{2}\beta^{2}\|\Psi_{T}\|_{\infty}b_{p}^{3}.
\end{aligned}
\]

Since $F_{p}(0) = 0$, Taylor's formula gives
\[
F_{p}(1) = F_{p}'(0) + \int_{0}^{1}(1 - z)F_{p}''(z)\rd{z},
\qquad
|F_{p}(1) - F_{p}'(0)|
\leq \frac{1}{2}\sup_{0 \leq z \leq 1}|F_{p}''(z)|.
\]
Combining the displayed second-derivative bound with the covariance replacement estimate yields
\[
\begin{aligned}
F_{p}(1)
&= \beta\sum_{j = 1}^{n}\EE_{U_{p}}G_{1}^{\times n}[K_{p,\alpha}(h_{k},h_{j})\Psi_{T}(\bb{h})] \\
&\qquad - n\beta\EE_{U_{p}}G_{1}^{\times(n + 1)}[K_{p,\alpha}(h_{k},h_{n + 1})\Psi_{T}(\bb{h})]
 + \OO(n^{2}\beta^{2}\|\Psi_{T}\|_{\infty}b_{p}^{3}),
\end{aligned}
\]
uniformly in the remaining prime phases. Taking expectation over those phases changes the partial expectation $\EE_{U_{p}}$ into $\EE_{\mathrm{H}}$. Summing over $p \leq X$ then restores $H_{\alpha} = \sum_{p \leq X}B_{p,\alpha}$ on the left and $K_{\alpha} = \sum_{p \leq X}K_{p,\alpha}$ in each covariance term on the right, with $G_{1} = G_{\beta,\sig,T}^{\mathrm{H}}$. By \eqref{eq:prime.block.bound}, the total error is bounded by
\[
Cn^{2}\beta^{2}\|\Psi_{T}\|_{\infty}\sum_{p \leq X}b_{p}^{3}
\leq Cn^{2}\beta^{2}\|\Psi_{T}\|_{\infty}.
\]
This proves \eqref{eq:haar.ibp}.
\end{proof}

\begin{proof}[Proof of Lemma~\ref{lem:covariance.input}]
Sum \eqref{eq:prime.covariance} over $p$. The contribution of all powers $j \geq 2$ is uniformly bounded by \eqref{eq:pnt.higher.powers}. Applying the truncated prime-sum estimate \eqref{eq:pnt.truncated.cosine} in Subsection~\ref{subsec:PNT} to the terms with $j = 1$ gives
\[
K_{\alpha}(h,h') = \frac{t}{2}\min\{\alpha,r_{T}(h,h')\} + \OO(\log{t}).
\]
Taking $\alpha = 1$ and $h = h'$ gives $V_{T} = t / 2 + \OO(\log{t})$. Dividing the case $\alpha = 1$ by this variance gives $R_{\sig,T}^{\mathrm{H}}(h,h') = r_{T}(h,h') + \OO((\log{t}) / t)$, and substitution proves \eqref{eq:truncated.overlap.covariance}. Since $|h - h'| \leq 2L_{T}$, the definition of $r_{T}$ gives the asserted lower bound.
\end{proof}

\subsection{Analytic estimates and Gaussian replacement}\label{subsec:replacement}

The calculation of the limiting perturbed free energy $f(u)$ in Lemma~\ref{lem:free.energy.input} requires replacing $H^{u}$ by a sum $\Gamma^{u} = \sum_{p}\Gamma_{p}^{u}$ of independent centered Gaussian prime contributions with matching covariances. The expected log-partition functions must differ by $\oo(t)$, so that division by $t$ makes this replacement error vanish. Concentration of the log-partition function around its expectation then gives convergence in probability after the same division. The spatial derivative estimate below also justifies approximating the integral defining the partition function by a finite sum of exponentials of values of the perturbed Haar field or its Gaussian replacement, with each value weighted by the length of its spatial grid interval. This approximation is quantified in \eqref{eq:grid.error} and is used in the proof of Lemma~\ref{lem:forest.comparison}. Lemma~\ref{lem:replacement.estimates} states the replacement, concentration and supremum estimates before deriving them from the prime-sum estimates in Subsection~\ref{subsec:PNT} and pointwise moment bounds.

Fix $q_{0} < \alpha < 1$, as in Lemma~\ref{lem:free.energy.input}, and let $u \in [ - 1/2,1/2]$. Let $B_{p}^{u} = B_{p} + uB_{p,\alpha}$ be the complete contribution of the prime $p$ to $H^{u}$, including all powers $p^{j}$ allowed by the corresponding cutoffs. Let $\Gamma_{p}^{u}$ be independent centered Gaussian fields satisfying
\begin{equation}\label{eq:cov.id}
\EE[\Gamma_{p}^{u}(h)\Gamma_{p}^{u}(h')] = \EE_{\mathrm{H}}[B_{p}^{u}(h)B_{p}^{u}(h')].
\end{equation}
For every prime $p$ and positive integer $j$, let $Z_{p,j}^{(1)}$ and $Z_{p,j}^{(2)}$ be centered real Gaussian variables of variance $1/2$, all independent across $(p,j)$ and the superscript index. Construct $\Gamma_{p}^{u}$ from $B_{p}^{u}$ by replacing $\Re(U_{p}^{j})$ with $\smash{Z_{p,j}^{(1)}}$ and $\Im(U_{p}^{j})$ with $\smash{Z_{p,j}^{(2)}}$, retaining the same deterministic coefficients, spatial frequencies $j\log{p}$, perturbation multipliers and cutoffs. Explicitly, replacing each factor $\Re(U_{p}^{j}e^{-\ii jh\log{p}})$ by $Z_{p,j}^{(1)}\cos(jh\log{p}) + Z_{p,j}^{(2)}\sin(jh\log{p})$ gives
\[
\Gamma_{p}^{u}(h)
= \sum_{j:p^{j}\leq X}\frac{1 + u\ind_{\{p^{j}\leq \exp((\log{X})^{\alpha})\}}}{j p^{j\sig}}
\left(Z_{p,j}^{(1)}\cos(jh\log{p}) + Z_{p,j}^{(2)}\sin(jh\log{p})\right).
\]
Orthogonality of the powers of each Haar phase, as used in \eqref{eq:prime.covariance}, shows that the real and imaginary Haar coefficients have the same pairwise second moments as their Gaussian replacements. Since their deterministic coefficients are unchanged, expanding the products on both sides of \eqref{eq:cov.id} gives the covariance identity. Recall the bound $\smash{b_{p} = \sum_{j \geq 1}1 / (j p^{j\sig})}$ from \eqref{eq:prime.block.bound} and the half-length $L_{T} = (\log T)^{\theta}$ of $I_{\theta}$. Define the log-partition function of a field $F$ by
\[
\Upsilon(F) = \log \int_{I_{\theta}} e^{\beta F(h)} \rd{h}.
\]
In particular, $\Upsilon(H^{u}) = tf_{T}(u)$ by \eqref{eq:perturbed.free.energy.definition}.

\begin{lemma}[Replacement, concentration, and supremum estimates]\label{lem:replacement.estimates}
Fix $\theta > - 1$, $\beta > 0$, and $q_{0} < \alpha < 1$, and let $(\sig,X)$ satisfy \eqref{eq:fixed.parameters}. Uniformly for $u \in [ - 1/2,1/2]$, Gaussian replacement satisfies, as $T \to \infty$,
\begin{equation}\label{eq:gaussian.replacement}
\left|\EE_{\mathrm{H}}[\Upsilon(H^{u})] - \EE\left[\Upsilon(\Gamma^{u})\right]\right|
\leq C\beta^{3}\sum_{p}b_{p}^{3} = \OO(1).
\end{equation}
For every deviation threshold $v > 0$, the log-partition function $\Upsilon(H^{u})$ of the perturbed Haar field obeys the concentration bound
\begin{equation}\label{eq:log.partition.function.concentration}
\PP_{\mathrm{H}}(|\Upsilon(H^{u}) - \EE_{\mathrm{H}}[\Upsilon(H^{u})]| > v)
\leq 2\exp( - cv^{2} / t).
\end{equation}
The same type of concentration bound holds for its Gaussian counterpart $\Upsilon(\Gamma^{u})$. Moreover,
\begin{equation}\label{eq:field.supremum}
\EE_{\mathrm{H}}\left[\sup_{h \in I_{\theta}}|H^{u}(h)|^{2}\right] \leq Ct^{2},
\qquad
\EE_{\mathrm{H}}\left[\sup_{h \in I_{\theta}}|(H^{u})'(h)|\right] \leq C\sqrt{t}\Lambda_{T}.
\end{equation}
Here the prime on $(H^{u})'$ denotes differentiation with respect to the spatial variable $h$. These supremum estimates also hold for the Gaussian field $\Gamma^{u}$. The constants may depend on $\theta$, $\beta$, $\alpha$, $\phi$, $\Phi$, and $C_{\sig}$, but are independent of $T$, the choice of $(\sig,X)$ within \eqref{eq:fixed.parameters}, and $u \in [ - 1/2,1/2]$.
\end{lemma}

\begin{proof}[Proof of Lemma~\ref{lem:replacement.estimates}]
All supremum norms below are over $h \in I_{\theta}$. The amplitude bound \eqref{eq:prime.block.bound} and Minkowski's inequality in $L^{3}$, applied to the Gaussian sine and cosine expansion, give
\[
\|B_{p}^{u}\|_{\infty} \leq Cb_{p}, \qquad \sum_{p}b_{p}^{3} = \OO(1), \qquad
\left(\EE[\|\Gamma_{p}^{u}\|_{\infty}^{3}]\right)^{1/3} \leq C\sum_{j:p^{j}\leq X}\frac{1}{j p^{j\sig}} \leq Cb_{p}.
\]
For bounded fields $F,V$, write $G_{F}$ for the Gibbs measure of $F$. The third derivative of $s \mapsto \Upsilon(F + sV)$ is $\beta^{3}$ times the third centered moment of $V$ under $G_{F + sV}$, so
\begin{equation}\label{eq:third.derivative.bound}
\left|\frac{\rd^{3}}{\rd s^{3}}\Upsilon(F + sV)\right| \leq 8\beta^{3}\|V\|_{\infty}^{3}.
\end{equation}
Take the Gaussian fields $\Gamma_{p}^{u}$ independent of all Haar phases, write $\EE_{*}$ for joint expectation, and replace complete prime contributions in increasing order. At prime $p$, the remaining field $F_{p}^{u} = \sum_{q < p}\Gamma_{q}^{u} + \sum_{p < q \leq X}B_{q}^{u}$ is independent of both $B_{p}^{u}$ and $\Gamma_{p}^{u}$. Conditional on $F_{p}^{u}$, their means vanish and their covariance kernels agree by \eqref{eq:cov.id}, so the expected linear and quadratic terms in the second-order Taylor expansions of $\Upsilon(F_{p}^{u} + B_{p}^{u})$ and $\Upsilon(F_{p}^{u} + \Gamma_{p}^{u})$ coincide. The remainder bound from \eqref{eq:third.derivative.bound} gives
\[
\left|\EE_{*}\left[\Upsilon(F_{p}^{u} + B_{p}^{u}) - \Upsilon(F_{p}^{u} + \Gamma_{p}^{u})\right]\right|
\leq \frac{4}{3}\beta^{3}\left(\EE_{\mathrm{H}}[\|B_{p}^{u}\|_{\infty}^{3}] + \EE[\|\Gamma_{p}^{u}\|_{\infty}^{3}]\right)
\leq C\beta^{3}b_{p}^{3}.
\]
Summing these successive differences proves \eqref{eq:gaussian.replacement}. Each replacement treats all powers of a prime together, preserving their dependence through the common Haar phase.

Changing only $U_{p}$ changes $H^{u}$ by at most $3b_{p}$ in supremum norm. Since $|\Upsilon(F) - \Upsilon(\widetilde{F})| \leq \beta \smash{\|F - \widetilde{F}\|_{\infty}}$, the corresponding change in $\Upsilon(H^{u})$ is at most $c_{p} = 3\beta b_{p}$, with $\sum_{p \leq X}c_{p}^{2} \leq C\beta^{2}t$ by \eqref{eq:prime.block.bound}. The bounded-differences inequality of \citet[Theorem~6.2, p.~166]{BoucheronLugosiMassart2013} therefore proves \eqref{eq:log.partition.function.concentration}.

For the Gaussian field, write $\Gamma^{u}(h) = \bb{a}^{u}(h)\cdot\bb{\xi}$ in terms of a vector $\bb{\xi}$ of independent standard Gaussian variables, and let $\mathcal{L}(\bb{z}) = \Upsilon(h \mapsto \bb{a}^{u}(h)\cdot\bb{z})$. Its gradient is a Gibbs average of $\beta\bb{a}^{u}(h)$, so covariance matching and \eqref{eq:prime.block.bound} give
\[
\|\nabla\mathcal{L}(\bb{z})\|_{2}^{2}
\leq \beta^{2}\sup_{h \in I_{\theta}}\|\bb{a}^{u}(h)\|_{2}^{2}
= \beta^{2}\sup_{h \in I_{\theta}}\Var_{\mathrm{H}}(H^{u}(h))
\leq \frac{9\beta^{2}}{4}\sum_{p \leq X}b_{p}^{2} \leq Ct.
\]
The Gaussian concentration inequality of \citet[Theorem~5.6, p.~123]{BoucheronLugosiMassart2013} now gives the Gaussian counterpart of \eqref{eq:log.partition.function.concentration}. Both bounds are uniform for $u \in [ - 1/2,1/2]$.

It remains to prove the supremum estimates. Independence, centering and Hoeffding's inequality, together with \eqref{eq:pnt.weighted.sums} and \eqref{eq:pnt.higher.powers}, give, uniformly in $h$, the following bounds in the probability $L^{2k}$ norm $\|\cdot\|_{2k}$ for every $k \in \N$:
\begin{equation}\label{eq:pointwise.moments}
\|H^{u}(h)\|_{2k} \leq C\sqrt{kt},
\qquad
\|(H^{u})'(h)\|_{2k} \leq C\sqrt{k}\Lambda_{T},
\qquad
\|(H^{u})''(h)\|_{2k} \leq C\sqrt{k}\Lambda_{T}^{2}.
\end{equation}
The corresponding Gaussian bounds follow from the Gaussian moment formula and the same variance estimates.

For a continuously differentiable function $f \in C^{1}([ - L_{T},L_{T}])$, the fundamental theorem of calculus gives
\[
\sup_{h \in I_{\theta}}|f(h)|^{2k} \leq |f(0)|^{2k} + 2k\int_{-L_{T}}^{L_{T}}|f(h)|^{2k - 1}|f'(h)|\rd{h}.
\]
Apply this inequality first to $f = H^{u}$ and then to $f = (H^{u})'$. After taking expectations, H\"{o}lder's inequality bounds the expected integrand by $\|f(h)\|_{2k}^{2k - 1}\|f'(h)\|_{2k}$. For the first application, \eqref{eq:pointwise.moments} consequently gives
\begin{equation}\label{eq:field.supremum.moment}
\left\|\sup_{h \in I_{\theta}}|H^{u}(h)|\right\|_{2k}
\leq C\sqrt{kt}\left(1 + \frac{kL_{T}\Lambda_{T}}{\sqrt{t}}\right)^{1/(2k)}.
\end{equation}
For the second application, the first and second derivative bounds give
\begin{equation}\label{eq:derivative.supremum.moment}
\left\|\sup_{h \in I_{\theta}}|(H^{u})'(h)|\right\|_{2k}
\leq C\sqrt{k}\Lambda_{T}\left(1 + kL_{T}\Lambda_{T}\right)^{1/(2k)}.
\end{equation}
Choose an integer $k\asymp t$. Since $L_{T}\Lambda_{T} \leq t^{-\phi}\exp((1 + \theta)t)$ by \eqref{eq:effective.cutoff.bounds}, the parenthesized factors raised to $1/(2k)$ are bounded uniformly in $T$ and in the choice of $(\sig,X)$. The bound \eqref{eq:field.supremum.moment}, together with $\|\cdot\|_{2} \leq \|\cdot\|_{2k}$ by H\"{o}lder's inequality, bounds the supremum $L^{2}$ norm by $Ct$, and \eqref{eq:derivative.supremum.moment}, together with $\|\cdot\|_{1} \leq \|\cdot\|_{2k}$, bounds the expected derivative supremum by $C\sqrt{t}\Lambda_{T}$. This proves \eqref{eq:field.supremum}. The same argument gives both estimates for the Gaussian field $\Gamma^{u}$.
\end{proof}

The following consequence of \eqref{eq:field.supremum} will be used in the proof of Lemma~\ref{lem:forest.comparison}. To make the spatial approximation explicit, fix $b > 1$ and partition $I_{\theta}$ into intervals $J_{1},\ldots,J_{N}$ of positive lengths at most $be^{-t}$. Choose $h_{i} \in J_{i}$ and use the actual lengths $|J_{i}|$ as weights, including for intervals meeting the endpoints of $I_{\theta}$. For this partition $\mathcal{P}$ and a field $F \in C^{1}(I_{\theta})$, define
\[
\Upsilon_{\mathcal{P}}(F) = \log\sum_{i = 1}^{N}|J_{i}|e^{\beta F(h_{i})}.
\]
For $h \in J_{i}$, the mean value theorem gives $|F(h) - F(h_{i})| \leq be^{-t}\sup_{I_{\theta}}|F'|$. Integrating the corresponding upper and lower bounds for $e^{\beta F(h)}$ over every $J_{i}$, summing and taking logarithms gives the deterministic inequality
\[
|\Upsilon(F) - \Upsilon_{\mathcal{P}}(F)| \leq \beta be^{-t}\sup_{I_{\theta}}|F'|.
\]
Taking expectations and applying \eqref{eq:field.supremum} proves
\begin{equation}\label{eq:grid.error}
\begin{aligned}
\EE_{\mathrm{H}}[|\Upsilon(H^{u}) - \Upsilon_{\mathcal{P}}(H^{u})|]
&\leq C\beta be^{-t}\sqrt{t}\Lambda_{T} = \oo(1), \\
\EE\left[\left|\Upsilon(\Gamma^{u}) - \Upsilon_{\mathcal{P}}(\Gamma^{u})\right|\right]
&\leq C\beta be^{-t}\sqrt{t}\Lambda_{T} = \oo(1).
\end{aligned}
\end{equation}
The errors vanish uniformly over \eqref{eq:fixed.parameters} because $e^{-t}\sqrt{t}\Lambda_{T} \leq t^{1/2 - \phi} \to 0$ by \eqref{eq:effective.cutoff.bounds} and $\phi > 1/2$. In particular, at $\sig = 1/2$ with $X = T^{1/t^{4}}$, the error bound is $\OO(t^{-7/2})$. The lower bound in Lemma~\ref{lem:forest.comparison} will use the same argument for disjoint intervals $J_{i} \subseteq I_{\theta}$ around a separated subset of grid points, without requiring those intervals to cover $I_{\theta}$. If $h_{i} \in J_{i}$ and $|J_{i}| \leq be^{-t}$, integration over their union gives
\begin{equation}\label{eq:grid.lower.bound}
\Upsilon(F) \geq \log\sum_{i}|J_{i}|e^{\beta F(h_{i})} - \beta be^{-t}\sup_{I_{\theta}}|F'|.
\end{equation}
The expected error is again $\oo(1)$ for the perturbed Haar field $H^{u}$ and the Gaussian field $\Gamma^{u}$.

\subsection{Gaussian forest comparisons}\label{subsec:forests}

Lemma~\ref{lem:forest.comparison} bounds the expected free energy of the perturbed Haar field between those of the two Gaussian forests illustrated in Figure~\ref{fig:gaussian.forest.comparison}. Its proof uses Gaussian replacement, spatial discretization and covariance comparison. The proof of Lemma~\ref{lem:forest.free.energy} then evaluates the limiting forest free energies through the variational expression \eqref{eq:forest.variational}.

\newpage
\begin{lemma}[Comparison with Gaussian forests]\label{lem:forest.comparison}
Fix $\theta > - 1$, $\beta > 0$, $q_{0} < \alpha < 1$, $u \in [ - 1/2,1/2]$ and a sufficiently large even integer $b$. There exist a full Gaussian forest and a retained Gaussian forest, each with $\exp((\theta \vee 0)t + \OO(1))$ roots. Each generation advances the logarithmic scale by $\log{b}$. The prescribed stage lengths and variance rates are
\[
\begin{aligned}
\text{stage one:}\quad &At = (\alpha - q_{0})t, \qquad &&\text{variance rate }(1 + u)^{2}/2, \\
\text{stage two:}\quad &Bt = (1 - \alpha)t, \qquad &&\text{variance rate }1/2.
\end{aligned}
\]
The stage boundary and terminal level are rounded to neighboring tree generations, changing each prescribed stage length by at most $2\log{b}$. The full forest has $b$ children per parent and the retained forest has $b/2$, so their branching entropies are, respectively,
\[
\frac{\log{b}}{\log{b}} = 1, \qquad \rho_{b} = \frac{\log(b / 2)}{\log{b}}.
\]
For $\mathrm{name} \in \{\mathrm{full},\mathrm{ret}\}$, let $\mathcal{L}_{t,b}^{\mathrm{name}}$ be the corresponding leaf set and let $X^{\mathrm{name}}(h)$ be the sum of the independent centered Gaussian edge increments along the path to leaf $h$. Define the forest free energy by
\[
F_{t,b}^{\mathrm{name}}(u) = \frac{1}{t}\log\left(e^{-t}\sum_{h \in \mathcal{L}_{t,b}^{\mathrm{name}}}e^{\beta X^{\mathrm{name}}(h)}\right).
\]
Then, as $T \to \infty$ with $b$ fixed,
\begin{equation}\label{eq:forest.comparison}
\EE[F_{t,b}^{\mathrm{ret}}(u)] - \oo(1) \leq \EE_{\mathrm{H}}[f_{T}(u)] \leq \EE[F_{t,b}^{\mathrm{full}}(u)] + \oo(1).
\end{equation}
\end{lemma}

\begin{figure}[!ht]
\centering
\includegraphics[width = 1\textwidth]{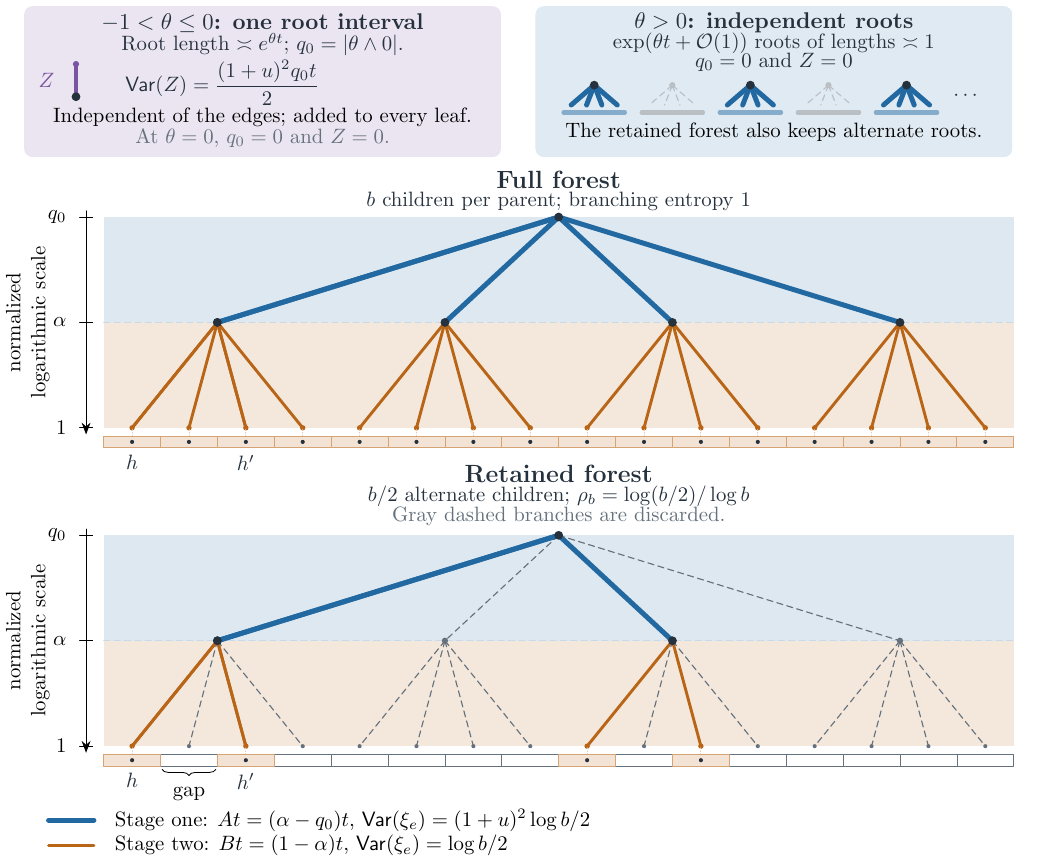}
\caption{Illustration of the full and retained Gaussian forests used in the comparison for $b = 4$. Each edge carries an independent centered Gaussian increment $\xi_{e}$. A leaf value is the sum along its path. Only two generations are shown; the stage depths are not to scale.}\label{fig:gaussian.forest.comparison}
\end{figure}

\begin{proof}[Proof of Lemma~\ref{lem:forest.comparison}]
Gaussian replacement, \eqref{eq:gaussian.replacement}, gives
\[
\left|\EE_{\mathrm{H}}[f_{T}(u)] - \frac{1}{t}\EE[\Upsilon(\Gamma^{u})]\right| = \OO(t^{-1}).
\]
It therefore suffices to compare $\Gamma^{u}$ with the two forests. Its covariance matches that of $H^{u}$, so \eqref{eq:perturbed.covariance} gives, uniformly for $h,h' \in I_{\theta}$,
\[
\begin{aligned}
\Cov(\Gamma^{u}(h),\Gamma^{u}(h')) &= \frac{t}{2}A_{u}(r_{T}(h,h')) + \OO(\log{t}), \\
A_{u}(r) &= (1 + u)^{2}\min\{r,\alpha\} + (r - \alpha)_{+}.
\end{aligned}
\]
The profile $A_{u}$ is nondecreasing and Lipschitz. The same covariance estimate holds after removing the Gaussian components corresponding to higher prime powers, since their covariance contribution is $\OO(1)$ by \eqref{eq:pnt.higher.powers}.

For the full forest, partition $I_{\theta}$ into equal root intervals of lengths comparable to $1$ when $\theta > 0$, and take $I_{\theta}$ itself as one root when $\theta \leq 0$. The number of roots is $\exp((\theta \vee 0)t + \OO(1))$. Subdivide each root into $b$ equal children, repeating until every terminal interval has length in $[e^{-t},be^{-t}]$. Denote these terminal intervals by $J_{i}$ and their centers by $h_{i}$. Assign independent centered Gaussian edge increments with the two variance rates in the statement. Rounding the stage boundary changes the covariance by $\OO(\log{b})$. When $\theta < 0$, add a common centered Gaussian $Z$, independent of the edges, to account for the scales below $q_{0}$:
\[
\Var(Z) = \frac{(1 + u)^{2}q_{0}t}{2}, \qquad
\log\sum_{i}|J_{i}|e^{\beta(X^{\mathrm{full}}(h_{i}) + Z)}
= \beta Z + \log\sum_{i}|J_{i}|e^{\beta X^{\mathrm{full}}(h_{i})}.
\]
Thus adding $Z$ leaves the expected free energy unchanged, and $\beta Z/t \to 0$ in $L^{2}$. For $\theta \geq 0$, let $Z = 0$.

If distinct leaves $h,h'$ in the same root share $k$ generations, define $r_{k} = \min\{1,q_{0} + k\log{b}/t\}$. The length of their common ancestor interval gives
\[
|h - h'| \leq Ce^{-q_{0}t}b^{-k}
\quad\Longrightarrow\quad
r_{T}(h,h') \geq r_{k} - \frac{C\log{b}}{t}.
\]
The covariance of the forest values, including $Z$, is $tA_{u}(r_{k})/2 + \OO(\log{b})$. Monotonicity and Lipschitz continuity of $A_{u}$ therefore give a lower bound on the actual Gaussian covariance. For distinct roots, which occur only when $\theta > 0$, the forest covariance is zero, while \eqref{eq:perturbed.covariance} and $A_{u}(r_{T}) \geq 0$ give $\Cov(\Gamma^{u}(h),\Gamma^{u}(h')) \geq - C\log{t}$. The diagonal variances agree up to $\OO(\log{t} + \log{b})$.

More precisely, let $\bb{X} = (\Gamma^{u}(h_{i}))_{i}$ and let $\bb{Y} = (X^{\mathrm{full}}(h_{i}) + Z)_{i}$ be independent vectors with covariance matrices $K_{X}$ and $K_{Y}$. The preceding estimates give
\[
K_{X}(i,j) - K_{Y}(i,j) \geq - E, \qquad
|K_{X}(i,i) - K_{Y}(i,i)| \leq E, \qquad
E = C(\log{t} + \log{b}).
\]
For the terminal interval lengths $|J_{i}| > 0$, define
\[
Z_{i}(s) = \sqrt{s}X_{i} + \sqrt{1 - s}Y_{i}, \qquad
p_{i}(s) = \frac{|J_{i}|e^{\beta Z_{i}(s)}}{\sum_{j}|J_{j}|e^{\beta Z_{j}(s)}}, \qquad 0 \leq s \leq 1.
\]
Gaussian integration by parts gives, for $0 < s < 1$,
\begin{equation}\label{eq:gaussian.log.partition.function.interpolation}
\begin{aligned}
\frac{\rd}{\rd s}\EE\left[\log\sum_{i}|J_{i}|e^{\beta Z_{i}(s)}\right]
&= \frac{\beta}{2}\EE\left[\sum_{i}p_{i}(s)\left(\frac{X_{i}}{\sqrt{s}} - \frac{Y_{i}}{\sqrt{1 - s}}\right)\right] \\
&= \frac{\beta^{2}}{2}\EE\left[\sum_{i}p_{i}(s)\big(K_{X}(i,i) - K_{Y}(i,i)\big)\right] \\
&\quad - \frac{\beta^{2}}{2}\EE\left[\sum_{i,j}p_{i}(s)p_{j}(s)\big(K_{X}(i,j) - K_{Y}(i,j)\big)\right].
\end{aligned}
\end{equation}
Since $\sum_{i}p_{i}(s) = \sum_{i,j}p_{i}(s)p_{j}(s) = 1$, the covariance bounds imply
\[
\frac{\rd}{\rd s}\EE\left[\log\sum_{i}|J_{i}|e^{\beta Z_{i}(s)}\right] \leq \beta^{2}E,
\qquad
\EE\left[\log\sum_{i}|J_{i}|e^{\beta X_{i}}\right]
\leq \EE\left[\log\sum_{i}|J_{i}|e^{\beta Y_{i}}\right] + \beta^{2}E.
\]
The second inequality follows by integration over $s \in [0,1]$. These are entrywise covariance comparisons; the difference $K_{X} - K_{Y}$ need not be a covariance matrix. Equation~\eqref{eq:grid.error} transfers the upper bound from the weighted grid sum to $\EE[\Upsilon(\Gamma^{u})]$, with error $\oo(1)$.

For the retained forest, keep alternate children at every subdivision, leaving $b/2$ children separated by discarded intervals. When $\theta > 0$, also retain alternate roots, changing their number by only a bounded factor. When $\theta \leq 0$, start from one central root interval of length comparable to $L_{T}$. If distinct retained leaves $h,h'$ in the same root share $k$ generations, their separation satisfies
\[
ce^{-q_{0}t}b^{-(k + 1)} \leq |h - h'| \leq Ce^{-q_{0}t}b^{-k}
\quad\Longrightarrow\quad
|r_{T}(h,h') - r_{k}| \leq \frac{C\log{b}}{t}.
\]
For leaves in distinct retained roots, the fixed gaps give $0 \leq r_{T}(h,h') \leq C/t$. Including $Z$ when $\theta < 0$ and treating the diagonal variances as above, the resulting covariance matrices satisfy
\[
\max_{i,j}|K_{X}(i,j) - K_{Y}(i,j)| \leq E,
\qquad E = C(\log{t} + \log{b}),
\]
where now $\bb{X}$ is evaluated at the retained grid points and $\bb{Y} = (X^{\mathrm{ret}}(h_{i}) + Z)_{i}$. Choose disjoint intervals $J_{i}$ around the retained grid points, contained in $I_{\theta}$ and of lengths $|J_{i}| = ce^{-t}$ for a sufficiently small fixed $c > 0$. Applying \eqref{eq:gaussian.log.partition.function.interpolation} in both directions yields
\[
\left|\EE\left[\log\sum_{i}|J_{i}|e^{\beta X_{i}}\right]
 - \EE\left[\log\sum_{i}|J_{i}|e^{\beta Y_{i}}\right]\right| \leq \beta^{2}E.
\]
Equation~\eqref{eq:grid.lower.bound} gives
\[
\EE[\Upsilon(\Gamma^{u})]
\geq \EE\left[\log\sum_{i}|J_{i}|e^{\beta X_{i}}\right] - \oo(1)
\geq \EE\left[\log\sum_{i}|J_{i}|e^{\beta Y_{i}}\right] - \beta^{2}E - \oo(1).
\]
This is the lower forest comparison.

In both constructions, the weights satisfy $ce^{-t} \leq |J_{i}| \leq be^{-t}$ for a fixed $c > 0$. Consequently,
\[
\left|\log\sum_{i}|J_{i}|e^{\beta Y_{i}} - \log\left(e^{-t}\sum_{i}e^{\beta Y_{i}}\right)\right|
\leq \max\{|\log{c}|,\log{b}\} = \OO(1 + \log{b}).
\]
Remove $Z$ in expectation and divide by $t$. The covariance comparison, weight replacement and Gaussian replacement errors are $\OO((1 + \log{t} + \log{b})/t) = \oo(1)$ for fixed $b$, proving \eqref{eq:forest.comparison}.
\end{proof}

\begin{proof}[Proof of Lemma~\ref{lem:forest.free.energy}]
For simplicity, assume that the two stage lengths are integer multiples of the block length used below; the general case requires only minor adjustments for rounding.

Write
\[
\mathcal{Z}_{t} = e^{-t}\sum_{h \in \mathcal{L}_{t}}e^{\beta X(h)}, \qquad F_{t} = \frac{1}{t}\log\mathcal{Z}_{t}, \qquad \log N_{0} = (\theta \vee 0)t + \OO(1),
\]
where $\mathcal{Z}_{t}$ is the partition function of the Gaussian forest, $F_{t}$ is the corresponding free energy, and $N_{0}$ is the number of roots. Here $X(h)$ sums the Gaussian increments over all generations of the two stages along the path to $h$, excluding the common trunk when $\theta < 0$. In that case, the sum of the increments along the common trunk is represented by the Gaussian variable $Z$ from Lemma~\ref{lem:forest.comparison}. Adding this common contribution to every leaf gives
\[
\frac{1}{t}\log\left(e^{-t}\sum_{h \in \mathcal{L}_{t}}e^{\beta(X(h) + Z)}\right) = F_{t} + \frac{\beta Z}{t}, \qquad \EE\left[\frac{\beta Z}{t}\right] = 0, \qquad \EE\left[\left(\frac{\beta Z}{t}\right)^{2}\right] = \OO(t^{-1}).
\]
Since $\beta Z/t$ tends to zero in $L^{2}$, the common trunk does not affect the claimed convergence in probability or in $L^{1}$. For the remainder of the proof, it therefore suffices to establish these limits for $F_{t}$.

For the upper bound, write $X(h) = X_{1}(h) + X_{2}(h)$, where $X_{1}(h)$ and $X_{2}(h)$ sum the increments over all generations of the first and second stages, respectively. For $\e > 0$, let $N_{1,t}(v_{1})$ count the vertices at the end of stage one satisfying $|X_{1} - Av_{1}t| \leq A\e t$, and let $N_{2,t}(v_{1},v_{2})$ count the leaves satisfying both
\[
|X_{1}(h) - Av_{1}t| \leq A\e t, \qquad |X_{2}(h) - Bv_{2}t| \leq B\e t.
\]
Along each fixed path, the two stage sums are independent and have variances $(1 + u)^{2}At/2$ and $Bt/2$, respectively. Subtracting the Gaussian costs of these constraints from the logarithmic growth rates of the numbers of available vertices and leaves gives the population exponents
\begin{equation}\label{eq:forest.bin.entropy}
\begin{aligned}
s_{1}(v_{1}) &= (\theta \vee 0) + \rho A - \frac{Av_{1}^{2}}{(1 + u)^{2}}, \\
s_{2}(v_{1},v_{2}) &= (\theta \vee 0) + \rho(A + B) - \frac{Av_{1}^{2}}{(1 + u)^{2}} - Bv_{2}^{2}.
\end{aligned}
\end{equation}
More precisely, multiplying the Gaussian interval probabilities by the corresponding numbers of vertices and leaves gives, uniformly for bounded $v_{1},v_{2}$,
\[
\begin{aligned}
\frac{1}{t}\log\EE[N_{1,t}(v_{1})] &= s_{1}(v_{1}) + \OO(\e) + \oo(1), \\
\frac{1}{t}\log\EE[N_{2,t}(v_{1},v_{2})] &= s_{2}(v_{1},v_{2}) + \OO(\e) + \oo(1).
\end{aligned}
\]
These expectation estimates use linearity and do not require independence between different paths. Here and below, $t \to \infty$ with the discretization parameters fixed.

Each stage sum has variance $\OO(t)$, and there are $\exp(\OO(t))$ leaves. A Gaussian union bound therefore gives, for sufficiently large fixed $M$,
\[
\PP\left(\max_{h \in \mathcal{L}_{t}}\max\left\{\frac{|X_{1}(h)|}{At},\frac{|X_{2}(h)|}{Bt}\right\} > M\right)
\leq \exp(Ct - cM^{2}t) \to 0.
\]
It consequently suffices to consider a finite grid of bins with half-width $\e$ covering $[ - M,M]^{2}$.

For $\eta > 0$, choose $C$ large enough for the preceding expectation estimates and let $\kappa = C\e + \eta$. Then, for $i \in \{1,2\}$ and all sufficiently large $t$,
\[
\EE[N_{i,t}] \leq e^{t(s_{i} + C\e + \eta/2)} = e^{t(s_{i} + \kappa - \eta/2)}.
\]
The arguments of $s_{i}$ and $N_{i,t}$ are suppressed here. If $s_{i} < - \kappa$, Markov's inequality gives
\[
\PP(N_{i,t} \geq 1) \leq \EE[N_{i,t}] \leq e^{-\eta t/2}.
\]
Thus bins with sufficiently negative population exponents are empty with probability tending to one. For every bin, Markov's inequality also bounds its population by
\[
\PP\left(N_{i,t} > e^{t(s_{i} + \kappa)}\right)
\leq \frac{\EE[N_{i,t}]}{e^{t(s_{i} + \kappa)}}
\leq e^{-\eta t/2}.
\]
Since there are only finitely many bins, with probability tending to one every occupied leaf bin satisfies $s_{1},s_{2} \geq - \kappa$ and contains at most $e^{t(s_{2} + \kappa)}$ leaves. Within a bin centered at $(v_{1},v_{2})$, each leaf also satisfies
\[
X(h) \leq t\left(Av_{1} + Bv_{2} + (A + B)\e\right).
\]
Combining this energy bound with the population bound controls the contribution of each bin to $\mathcal{Z}_{t}$. Summing over the fixed number of bins contributes only $\oo(1)$ after taking the logarithm and dividing by $t$, so
\[
F_{t} \leq - 1 + \sup_{\substack{|v_{1}|,|v_{2}| \leq M + \e \\ s_{1}(v_{1}),s_{2}(v_{1},v_{2}) \geq - \kappa}}
\left\{s_{2}(v_{1},v_{2}) + \beta(Av_{1} + Bv_{2})\right\}
+ \kappa + \beta(A + B)\e + \oo(1)
\]
with probability tending to one. Replacing $v_{i}$ by $|v_{i}|$ preserves $s_{1},s_{2}$ and increases the linear term. As $\e,\eta \downarrow 0$, compactness of the bounded region containing the bin centers and continuity of the constraints show that any limiting admissible pair belongs to $\mathcal{D}_{\rho}$. Hence
\[
\PP\left(F_{t} > \mathcal{F}_{\rho}(u) + a\right) \to 0, \qquad a > 0,
\]
where $\mathcal{F}_{\rho}(u)$ is the right-hand side of \eqref{eq:forest.variational}.

For the lower bound, fix a strictly feasible pair $(v_{1},v_{2})$, so that
\[
s_{1} = s_{1}(v_{1}) > 0, \qquad s_{2} = s_{2}(v_{1},v_{2}) > 0.
\]
Strictly feasible pairs are dense in $\mathcal{D}_{\rho}$: multiplying any feasible pair by a factor less than one makes both constraints strict, and this factor may tend to one. The aim is to find sufficiently many leaves whose increments remain close to these prescribed rates throughout both stages.

Group every $K$ consecutive generations of the original forest into a block of scale length $\Delta = K\log{b}$.

Starting with all $N_{0}$ roots, construct a selected population at successive block boundaries. From each selected parent, keep only those descendants at the next block boundary whose stage-$i$ block increment lies in $[(v_{i} - \e)\Delta,(v_{i} + \e)\Delta]$. Let $Z_{i}^{\mathrm{off}}$ be the number kept from one parent. Each parent has $e^{\rho\Delta}$ descendants after the $K$ original generations that make up one block. For a block in stage $i$, the increment used in the selection rule is the sum of the $K$ Gaussian edge increments along the path from a parent at the beginning of the block to a descendant at its end. This sum is Gaussian with mean zero and variance $w_{i}^{2}\Delta/2$, where $w_{1} = 1 + u$ and $w_{2} = 1$. Linearity of expectation therefore gives
\[
\mu_{i} \leqdef \EE[Z_{i}^{\mathrm{off}}] = e^{\rho\Delta}\PP\left(|G_{i} - v_{i}\Delta| \leq \e\Delta\right), \qquad G_{i} \sim \mathcal{N}\left(0,\frac{w_{i}^{2}\Delta}{2}\right).
\]
This calculation does not require independence between the increments of descendants of the same parent.

Throughout the prescribed interval $[(v_{i} - \e)\Delta,(v_{i} + \e)\Delta]$, the exponent in the Gaussian density equals $- v_{i}^{2}\Delta/w_{i}^{2} + \OO(\e\Delta)$. Multiplying the corresponding upper and lower density bounds by the interval length $2\e\Delta$ gives
\[
\PP\left(|G_{i} - v_{i}\Delta| \leq \e\Delta\right)
= \frac{2\e\sqrt{\Delta}}{w_{i}\sqrt{\pi}}
\exp\left(- \frac{v_{i}^{2}\Delta}{w_{i}^{2}} + \OO(\e\Delta)\right).
\]
For fixed $\e > 0$, the logarithm of the prefactor is $\OO(\log\Delta)$ and hence $\oo_{K}(\Delta)$ as $K \to \infty$. Taking logarithms in the formula for $\mu_{i}$ and dividing by $\Delta$ consequently yields
\[
\frac{\log\mu_{i}}{\Delta} = \rho - \frac{v_{i}^{2}}{w_{i}^{2}} + \OO(\e) + \oo_{K}(1).
\]

Although each parent has a deterministic number of descendants in the original forest, the number satisfying the increment condition is random. These selected descendants serve as its children in the selected process. Thus one generation of the selected process corresponds to $K$ generations of the original forest. Distinct parents at the same block boundary have disjoint descendant trees, whose Gaussian increments are independent of each other and of the preceding blocks. Within each stage, the offspring counts therefore have the same distribution and are independent. The selected population in stage $i$ is consequently a Galton--Watson process with offspring mean $\mu_{i}$ and offspring variance $\Var(Z_{i}^{\mathrm{off}})$.

To control the size of this population, consider a Galton--Watson process $(Z_{n})_{n\geq 0}$ with deterministic initial size $N \geq 1$, offspring mean $\mu > 0$ and finite offspring variance $\sig_{Z}^{2}$. Conditional on $Z_{n}$, the next population is the sum of $Z_{n}$ independent offspring counts. Taking expectations and applying the conditional variance formula gives
\[
\EE[Z_{n}] = N\mu^{n}, \qquad \Var(Z_{n + 1}) = \mu^{2}\Var(Z_{n}) + \sig_{Z}^{2}N\mu^{n};
\]
see, e.g., the moment calculations in \citet[Lemma~4.3.9 and Example~4.4.9]{Durrett2019}. Divide the variance recurrence by $N^{2}\mu^{2n + 2}$ and sum over the generations, starting from $\Var(Z_{0}) = 0$. Evaluating the resulting geometric sum gives
\begin{equation}\label{eq:branching.variance}
\frac{\Var(Z_{n})}{(\EE[Z_{n}])^{2}}
=
\begin{cases}
\dfrac{\sig_{Z}^{2}(1 - \mu^{-n})}{N\mu(\mu - 1)}, & \mu \neq 1, \\[4mm]
\dfrac{n\sig_{Z}^{2}}{N}, & \mu = 1.
\end{cases}
\end{equation}
In particular,
\begin{equation}\label{eq:branching.variance.bounds}
\frac{\Var(Z_{n})}{(\EE[Z_{n}])^{2}} \leq C(1 + n)\left(\frac{1}{N} + \frac{1}{N\mu^{n}}\right), \qquad
\frac{\Var(Z_{n})}{(\EE[Z_{n}])^{2}} \leq \frac{C}{N} \quad\text{if }\mu > 1.
\end{equation}
The constants may depend on the offspring distribution. In the application below, this distribution is fixed once $\e$ and $K$ have been chosen, and its variance is finite because the offspring count is bounded by $e^{\rho\Delta}$.

Let $n_{1} = At/\Delta$ and $n_{2} = Bt/\Delta$ be the numbers of blocks in the two stages, and define
\[
\widetilde{s}_{1} = (\theta \vee 0) + A\frac{\log\mu_{1}}{\Delta}, \qquad
\widetilde{s}_{2} = \widetilde{s}_{1} + B\frac{\log\mu_{2}}{\Delta}.
\]
Here $\widetilde{s}_{1}, \widetilde{s}_{2}$ are the growth exponents of the expected selected populations at the ends of the two stages, to be compared with the theoretical target exponents $s_{1}, s_{2}$ as $\e \downarrow 0$ and $K \to \infty$. Fix $0 < \eta < \min\{s_{1},s_{2}\}/2$. The estimate for $\log\mu_{i}/\Delta$ allows $\e$ to be chosen sufficiently small and then $K$ sufficiently large that
\[
|\widetilde{s}_{i} - s_{i}| < \eta, \qquad \widetilde{s}_{i} > 0, \qquad i \in \{1,2\}.
\]
Keep these choices fixed while letting $t \to \infty$.

Write $Y_{1}$ for the selected population at the end of stage one and $m_{1} = N_{0}\mu_{1}^{n_{1}}$ for its mean. The expected final population is $m_{1}\mu_{2}^{n_{2}}$, and
\[
\frac{1}{t}\log m_{1} = \widetilde{s}_{1} + \oo(1), \qquad
\frac{1}{t}\log\left(m_{1}\mu_{2}^{n_{2}}\right) = \widetilde{s}_{2} + \oo(1).
\]

If $\theta \leq 0$, then $\widetilde{s}_{1} = A\log\mu_{1}/\Delta > 0$, so $\mu_{1} > 1$. The second variance bound in \eqref{eq:branching.variance.bounds} and the Paley-Zygmund inequality give
\[
\PP(Y_{1} \geq m_{1}/2) \geq \frac{m_{1}^{2}}{4\EE[Y_{1}^{2}]} \geq c > 0.
\]
If $\theta > 0$, both $N_{0}$ and $m_{1}$ grow exponentially, so the first variance bound in \eqref{eq:branching.variance.bounds} gives
\[
\frac{\Var(Y_{1})}{m_{1}^{2}}
\leq C(1 + n_{1})\left(N_{0}^{-1} + m_{1}^{-1}\right) \to 0.
\]
Chebyshev's inequality then implies $\PP(Y_{1} \geq m_{1}/2) \to 1$. Thus, in either case, the first stage produces at least half its expected population with probability bounded away from zero.

Let $Y_{2}$ be the final selected population. Conditional on $Y_{1}$, the second stage consists of $n_{2}$ generations of the selected process starting from $Y_{1}$ individuals, and its conditional mean is $Y_{1}\mu_{2}^{n_{2}}$. On $\{Y_{1} \geq m_{1}/2\}$, Chebyshev's inequality and \eqref{eq:branching.variance} give
\[
\PP\left(Y_{2} < \frac{Y_{1}\mu_{2}^{n_{2}}}{2}\,\middle|\,Y_{1}\right)
\leq C(1 + n_{2})\left(\frac{1}{Y_{1}} + \frac{1}{Y_{1}\mu_{2}^{n_{2}}}\right)
\leq C(1 + n_{2})\left(\frac{1}{m_{1}} + \frac{1}{m_{1}\mu_{2}^{n_{2}}}\right) \to 0.
\]
The last bound tends to zero because $n_{2} = \OO(t)$, whereas both $m_{1}$ and $m_{1}\mu_{2}^{n_{2}}$ grow exponentially. This argument also covers $\mu_{2} \leq 1$. Consequently, for some $c > 0$ and all sufficiently large $t$, $\PP(Y_{2} \geq m_{1}\mu_{2}^{n_{2}}/4) \geq c$. On this event, the population size satisfies
\[
\frac{1}{t}\log Y_{2} \geq s_{2} - \eta + \oo(1).
\]
Every selected leaf also satisfies the increment condition in each block along its path, so
\[
\frac{X(h)}{t}
\geq \frac{n_{1}\Delta}{t}(v_{1} - \e) + \frac{n_{2}\Delta}{t}(v_{2} - \e)
= Av_{1} + Bv_{2} - (A + B)\e.
\]
Their total contribution to the partition function is therefore at least
\[
\mathcal{Z}_{t}
\geq e^{-t}Y_{2}\exp\left(\beta t\left(Av_{1} + Bv_{2} - (A + B)\e\right)\right).
\]
Taking logarithms, dividing by $t$ and using the lower bound on $Y_{2}$ gives
\begin{equation}\label{eq:forest.positive.probability}
\PP\left(F_{t} \geq - 1 + s_{2} + \beta(Av_{1} + Bv_{2}) - 2\eta - \beta(A + B)\e\right) \geq c
\end{equation}
for all sufficiently large $t$, with the additional $\eta$ absorbing the terms tending to zero.

Applying the Gaussian concentration argument from the proof of Lemma~\ref{lem:replacement.estimates} to the forest log-partition function gives
\[
\|\nabla\log\mathcal{Z}_{t}\|^{2} \leq \beta^{2}\max_{h \in \mathcal{L}_{t}}\Var(X(h)) \leq Ct, \qquad
\PP\left(|F_{t} - \EE[F_{t}]| > a\right) \leq 2e^{-ca^{2}t}, \quad a > 0.
\]
If $\EE[F_{t}] < L$, the event $\{F_{t} \geq L\}$ requires a deviation of at least $L - \EE[F_{t}]$ from the mean. Thus, for a fixed threshold $L$ and $p \in (0,1]$,
\[
\PP(F_{t} \geq L) \geq p
\quad\Longrightarrow\quad
\EE[F_{t}] \geq L - \sqrt{\frac{\log(2/p)}{ct}}.
\]
Applying this implication to the preceding estimate \eqref{eq:forest.positive.probability} for $F_{t}$ yields, for the fixed choices of $\eta,\e$ and $K$,
\[
\liminf_{t\to\infty}\EE[F_{t}]
\geq - 1 + s_{2}(v_{1},v_{2}) + \beta(Av_{1} + Bv_{2}) - 2\eta - \beta(A + B)\e.
\]

For each sufficiently small $\eta > 0$, choose $\e$ sufficiently small and then $K$ sufficiently large as above, keeping both fixed while $t \to \infty$. Letting $\eta \downarrow 0$, with $\e$ also tending to zero, gives
\[
\liminf_{t\to\infty}\EE[F_{t}]
\geq - 1 + s_{2}(v_{1},v_{2}) + \beta(Av_{1} + Bv_{2}).
\]
Taking the supremum over strictly feasible pairs, which are dense in $\mathcal{D}_{\rho}$, yields
\[
\liminf_{t\to\infty}\EE[F_{t}] \geq \mathcal{F}_{\rho}(u).
\]
For any fixed $a > 0$, the mean is therefore at least $\mathcal{F}_{\rho}(u) - a/2$ for all sufficiently large $t$. Gaussian concentration then implies
\[
\PP\left(F_{t} < \mathcal{F}_{\rho}(u) - a\right) \to 0.
\]
Together with the upper bound, this proves convergence in probability.

Finally, $|\mathcal{L}_{t}| = \exp(\OO(t))$ and $\max_{h}\Var(X(h)) = \OO(t)$. The Gaussian tail bound and a union bound over the leaves give
\begin{equation}\label{eq:forest.maximum.tail}
\PP\left(\max_{h \in \mathcal{L}_{t}}|X(h)| > x\right)
\leq \min\left\{1,2\exp\left(Ct - \frac{cx^{2}}{t}\right)\right\}, \qquad x > 0.
\end{equation}
Set $x_{0} = t\sqrt{(C + 1)/c}$. Then \eqref{eq:forest.maximum.tail} gives
\[
\begin{aligned}
\EE\left[\max_{h \in \mathcal{L}_{t}}|X(h)|^{2}\right]
&= 2\int_{0}^{\infty}x\PP\left(\max_{h \in \mathcal{L}_{t}}|X(h)| > x\right)\rd{x} \\
&\leq 2\int_{0}^{x_{0}}x\rd{x} + 4e^{Ct}\int_{x_{0}}^{\infty}xe^{-cx^{2}/t}\rd{x} \\
&= x_{0}^{2} + \frac{2t}{c}\exp\left(Ct - \frac{cx_{0}^{2}}{t}\right)
= \frac{C + 1}{c}t^{2} + \frac{2t}{c}e^{-t}
= \OO(t^{2}).
\end{aligned}
\]
Bounding every summand in $\mathcal{Z}_{t}$ between $e^{-\beta\max_{h}|X(h)|}$ and $e^{\beta\max_{h}|X(h)|}$ also gives
\[
|F_{t}| \leq \left|\frac{1}{t}\log\left(e^{-t}|\mathcal{L}_{t}|\right)\right| + \frac{\beta}{t}\max_{h \in \mathcal{L}_{t}}|X(h)|.
\]
The first term is bounded, and the preceding maximum estimate controls the second in $L^{2}$. Thus $\limsup_{t\to\infty}\EE[|F_{t}|^{2}] < \infty$, so uniform integrability upgrades convergence in probability to convergence in $L^{1}$.
\end{proof}

\section*{Funding}
\addcontentsline{toc}{section}{Funding}

F.\ Ouimet is supported by the Fonds de recherche du Québec -- Nature et technologies (FRQNT) through the Relève professorale grant doi:10.69777/377977, and by the Natural Sciences and Engineering Research Council of Canada (NSERC) through Discovery Grant RGPIN-2026-04471 and Discovery Launch Supplement DGECR-2026-00449.

\renewcommand{\bibsection}{\section*{References}\addcontentsline{toc}{section}{References}}
\setlength{\bibsep}{0pt plus 0ex}

\bibliographystyle{authordate1-doi-mr}
\bibliography{bib}

\end{document}